\documentclass[11pt]{article}

\usepackage[margin=1in]{geometry}
\usepackage[numbers]{natbib}
\usepackage[utf8]{inputenc}
\usepackage[T1]{fontenc}
\usepackage{hyperref}
\usepackage{url}
\usepackage{booktabs}
\usepackage{amsfonts}
\usepackage{nicefrac}
\usepackage{microtype}
\usepackage{xcolor}
\usepackage{preamble}

\title{Near-Optimal Decentralized Stochastic Convex Optimization \\ over Networks}

\renewcommand{\thefootnote}{\fnsymbol{footnote}}

\author{
  Nitai Kluger\footnotemark[1]{\;\,}\footnotemark[2] \qquad
  Amit Attia\footnotemark[1]{\;\,}\footnotemark[2] \qquad
  Tomer Koren\footnotemark[3]
}
\date{}

\begin{document}
\maketitle

{\renewcommand{\footnotesize}{\scriptsize}
\footnotetext[1]{Equal contribution.}
\footnotetext[2]{Blavatnik School of Computer Science and AI, Tel Aviv University; \texttt{\{nitaikluger,amitattia\}@mail.tau.ac.il}}
\footnotetext[3]{Blavatnik School of Computer Science and AI, Tel Aviv University, and Google Research Tel Aviv; \texttt{tkoren@tauex.tau.ac.il}}
}

\renewcommand{\thefootnote}{\arabic{footnote}}
\setcounter{footnote}{0}

\begin{abstract}
We study decentralized stochastic smooth convex optimization, where $M$ workers (machines) minimize an average objective using local stochastic gradients and neighbor-only communication over a fixed gossip network.
A central question in this setting is to determine the largest number of workers that can be used under a total budget of $N$ gradient samples while still preserving the centralized $O(1/\sqrt N)$ statistical rate.
We introduce an accelerated decentralized method that preserves this rate for up to $\smash{M\lesssim \sqrt{\rho}\,N^{3/4}}$ workers, where $\rho$ is the spectral gap of the gossip network, improving the best prior maximal scaling of $\smash{M\lesssim \rho\sqrt N}$.
The method is based on a one-step-delayed stochastic acceleration scheme that enables workers to interleave minibatching with accelerated gossip while controlling residual disagreement, and its guarantee depends only logarithmically on the optimum-local heterogeneity.
We establish a matching lower bound for linear-span decentralized first-order methods, showing that the method is optimal up to logarithmic factors.
\end{abstract}

\section{Introduction}

Distributed optimization has become a central paradigm in modern large-scale learning and inference.  In many applications, data and computation are spread across machines, so optimization must combine many parallel gradient samples with limited communication.  The central promise is statistical: for a fixed budget of \(N\) gradient samples, using \(M\) machines should reduce the number of sequential rounds without degrading the accuracy attainable from those samples.

A useful way to make this promise precise is the \emph{critical parallelism} threshold: the maximum number of machines \(M\) for which parallelism preserves the centralized statistical rate.  Beyond this threshold, the remaining optimization or communication error is no longer hidden by sampling noise, and additional machines yield diminishing returns for sample efficiency.  For smooth convex stochastic optimization, centralized synchronized minibatch SGD preserves the statistical rate for \(\smash{M\lesssim \sqrt N}\) workers, while centralized accelerated minibatching raises this scale to \(\smash{M\lesssim N^{3/4}}\)~\citep{dekel2012optimal,lan2012optimal}.  This \(\smash{N^{3/4}}\)-scale is therefore the natural benchmark for optimization parallelism.

In this paper we study the same question for decentralized stochastic optimization over networks.  In contrast to parameter-server or all-reduce models, decentralized methods communicate only along the edges of a network, formalized as an undirected graph over distributed nodes, typically through gossip averaging, i.e., repeated weighted averaging of vectors between neighbors.  This removes a central bottleneck and is a natural abstraction for federated, peer-to-peer, and geographically distributed systems.  The price is that workers no longer hold identical iterates: gradients are sampled at different local points, and the disagreement between these points appears as a bias in the stochastic oracle seen by the network average.

This bias is especially delicate under data heterogeneity.  When different workers see different local objectives, their local gradients can point in systematically different directions even when evaluated at the same, nearly optimal model.  Thus residual network disagreement is not merely an implementation error; it can create a persistent bias that accelerated methods are known to amplify if it is not controlled carefully~\citep{devolder2014first}.

Modern distributed training systems often allow stochastic-gradient computation
and neighbor communication to proceed concurrently on the same worker.  This
motivates the nonblocking scheduling model studied here, where each wall-clock
round contains one gradient computation alongside a small, constant number of
local gossip steps.  Blocking schemes, which spend many consecutive rounds solely
on gossiping, may improve agreement, but they idle the stochastic oracles and waste
scarce sequential rounds without collecting gradient samples.
The relevant challenge is therefore not only to communicate efficiently, but to
obtain enough network mixing while continuous parallel sampling is underway.

Several prior works have tackled decentralized stochastic convex optimization
using gossip-based stochastic gradients, gradient tracking, and anytime
averaging~\citep{koloskova2020unified,koloskova2021improved,eisen2025enhancing}.
To date, the best critical-parallelism guarantee is due to
\citet{eisen2025enhancing}: their method preserves the centralized statistical
rate for 
\(
    \smash{M \lesssim \rho\sqrt N}
\) 
workers, where \(\rho\) is the spectral gap of the network graph.
Thus, on well-connected networks \((\rho=\Theta(1))\), it matches the
\(\smash{\sqrt N}\)-scale of centralized synchronized minibatch SGD, while on general
networks it loses a factor depending on the spectral gap.

This leaves a considerable gap to the accelerated \(N^{3/4}\)-scale.  Under a fixed sample
budget, the algorithm has only \(N/M\) distributed rounds, so the deterministic
optimization and network errors must decay fast enough to remain below the
\(O(\smash{1/\sqrt N})\) statistical term.  Non-accelerated stochastic approximation
does not provide this decay, and the spectral-gap dependence in existing
stochastic decentralized bounds is likewise non-accelerated relative to the
optimal deterministic network rates of \citet{scaman2017optimal}.  
A separate obstacle is data heterogeneity: the method of \citet{eisen2025enhancing}
suffers from a polynomial dependence on heterogeneity, which can dominate the
rate.  Gradient-tracking methods can remove explicit heterogeneity terms in
some regimes~\citep{nedic2017achieving,koloskova2021improved}, but the known
stochastic convex bounds imply worse critical parallelism in the sample budget
\(N\).

In this paper, we overcome these limitations with a novel delayed accelerated stochastic
method that combines two acceleration mechanisms.  The optimization dynamics
use a Nesterov-type accelerated stochastic recursion, while the network uses
accelerated gossip to mix local copies.  A one-step-delayed, momentum-compensated
query schedule ties these mechanisms together: workers keep sampling gradients
at an already available query point while accelerated gossip prepares the next
query point from the current local copies.  Hence, stochastic acceleration and
network mixing are interleaved rather than separated into blocking phases, and
the remaining disagreement enters the averaged dynamics as a controlled oracle
bias. The method thus performs a single gossip communication step in each round, in addition to one local stochastic gradient query at each node.

The resulting rate is optimal up to logarithmic factors within the standard class of deterministic linear-span (decentralized) first-order methods, with stochastic gradients as the
only source of algorithmic randomness.  It preserves the optimal stochastic
term, attains the optimal network dependence, and depends only logarithmically
on the data heterogeneity at the optimum.  This gives, to the best of our knowledge, the
first accelerated decentralized method in this smooth convex stochastic setting
that is rate-optimal up to logarithmic factors.

\begin{table}[t]
\centering
\small
% \textbf{Schematic comparison of convex decentralized stochastic optimization rates.}
\par\vspace{0.4em}
\setlength{\tabcolsep}{4pt}
\renewcommand{\arraystretch}{1.2}
\begin{tabular}{>{\raggedright\arraybackslash}m{0.27\textwidth}>{\raggedright\arraybackslash}m{0.36\textwidth}>{\raggedright\arraybackslash}m{0.26\textwidth}}
\hline\noalign{\vskip 0.3em}
Method & Iteration complexity to \(\epsilon\)-accuracy & Maximal parallelism \\
\noalign{\vskip 0.3em}\hline\hline\noalign{\vskip 0.3em}
\begin{tabular}[c]{@{}l@{}}
D-SGD\\
\citep{koloskova2020unified}
\end{tabular}
&
\(\displaystyle
    \frac{\sigma^2}{M\epsilon^2}
    +\frac{\sigma}{\sqrt{\rho}\,\epsilon^{3/2}}
    +\frac{\zeta_\star}{\rho\,\epsilon^{3/2}}
    +\frac{1}{\rho\epsilon}
\)
&
\begin{tabular}[c]{@{}>{\raggedright\arraybackslash}p{0.26\textwidth}@{}}
\(\Theta(\sqrt{\rho}\,N^{1/4})\) if \(\zeta_\star \approx 0\); \\
\(\Theta(\rho N^{1/4})\) otherwise
\end{tabular}
\\
\noalign{\vskip 0.3em}\hline\noalign{\vskip 0.3em}
\begin{tabular}[c]{@{}l@{}}
Gradient Tracking\\
\citep{koloskova2021improved}
\end{tabular}
&
\(\displaystyle
    \frac{\sigma^2}{M\epsilon^2}
    +\frac{\sigma}{c\sqrt{\rho}\,\epsilon^{3/2}}
    +\frac{1}{c\rho\epsilon}
\)
&
\begin{tabular}[c]{@{}>{\raggedright\arraybackslash}p{0.26\textwidth}@{}}
\(\widetilde \Theta(\sqrt{\rho}\,N^{1/4})\) \\
(treating \(c=\Theta(1)\))
\end{tabular}
\\
% \noalign{\vskip 0.3em}\hline\noalign{\vskip 0.3em}
% \begin{tabular}[c]{@{}l@{}}
% \(D^2\)\\
% \citep{tang2018d2}
% \end{tabular}
% &
% \(\displaystyle
%     \frac{\sigma^2}{M\epsilon^2}
%     +\frac{\sigma}{\sqrt{\rho}\,\epsilon^{3/2}}
%     +\frac{1}{\rho\epsilon}
% \)
% &
% \begin{tabular}[c]{@{}>{\raggedright\arraybackslash}p{0.26\textwidth}@{}}
% \(\widetilde \Theta(\sqrt{\rho}\,N^{1/4})\) \\
% (requires \(\lambda_M(P) > -\tfrac13\))
% \end{tabular}
% \\
\noalign{\vskip 0.3em}\hline\noalign{\vskip 0.3em}
\begin{tabular}[c]{@{}l@{}}
DAT-SGD\\
\citep{eisen2025enhancing}
\end{tabular}
&
\(\displaystyle
    \frac{\sigma^2}{M\epsilon^2}
    +\frac{\sqrt{\sigma}+\sqrt{\zeta}}{\rho\epsilon}
    +\frac{1}{\epsilon}
\)
&
\begin{tabular}[c]{@{}>{\raggedright\arraybackslash}p{0.26\textwidth}@{}}
\(\Theta(\rho\sqrt N)\)
\end{tabular}
\\
\noalign{\vskip 0.3em}\hline\hline\noalign{\vskip 0.3em}
\begin{tabular}[c]{@{}l@{}}
\dadq \\
{\bfseries This work} (\cref{thm:distributed-gossip})
\end{tabular}
&
\(\displaystyle
    \frac{\sigma^2}{M\epsilon^2}
    +\frac{1}{\sqrt{\rho\epsilon}} \log\frac{N\zeta_\star}{\rho}
\)
% up to logarithmic dependence on \(\zeta_\star\)
&
\begin{tabular}[c]{@{}>{\raggedright\arraybackslash}p{0.26\textwidth}@{}}
\(\widetilde \Theta(\sqrt{\rho}\,N^{3/4})\)
\end{tabular}
\\
\noalign{\vskip 0.3em}\hline\noalign{\vskip 0.3em}
\begin{tabular}[c]{@{}l@{}}
Lower bound\\
{\bfseries This work} (\cref{thm:linear-span-lower-bound})
\end{tabular}
&
\(\displaystyle
    \frac{\sigma^2}{M\epsilon^2}
    +\frac{1}{\sqrt{\rho\epsilon}}
\)
&
\begin{tabular}[c]{@{}>{\raggedright\arraybackslash}p{0.26\textwidth}@{}}
No method can exceed \(\Theta(\sqrt{\rho}\,N^{3/4})\)
\end{tabular}
\\
\noalign{\vskip 0.3em}\hline\noalign{\vskip 1.0em}
\end{tabular}
\caption{Comparison of rates for convex decentralized stochastic optimization.
Here \(N\) is the total number of stochastic gradient samples, $M$ is the number of distributed nodes, \(\rho\) is the spectral gap, \(\sigma\) denotes the local noise scale, \(\zeta_\star\) denotes data heterogeneity measured at the optimum, \(\zeta\) denotes a global heterogeneity bound, and \(c\) is a negative-eigenvalue parameter appearing in the gradient-tracking analysis of \citet{koloskova2021improved}. The table suppresses numerical constants, smoothness/radius parameters, and most logarithmic factors.}
\end{table}

\subsection{Summary of contributions}

Our main contributions are summarized as follows:
\begin{enumerate}[label=(\roman*),leftmargin=*]
    \item We propose \dadq, an accelerated method for decentralized stochastic smooth convex optimization over a fixed gossip network, and prove an optimal finite-sample rate for it, up to logarithmic factors.
    % The method combines accelerated gossip with a one-step-delayed accelerated stochastic scheme, allowing each worker to interleave communication for the current state with minibatch gradient collection at already available delayed query points.  
    For a total budget of $N$ stochastic-gradient computations over a network with spectral gap $\rho \in (0,1]$, the output $\widetilde x_{T,i}$ of each individual node attains convergence rate
    \[
        \E[f(\widetilde x_{T,i})-f(x^\star)]
        \lesssim
        \frac{\sigma R}{\sqrt N}
        +
        \frac{LR^2M^2}{\rho N^2}
        \cdot
        \operatorname{polylog}\!\left(
            \frac{N}{\rho},
            \frac{\zeta_\star+\sigma}{LR}
        \right),
    \]
    where $R$ is an upper bound on the initial distance to the solution, $\sigma^2$
    bounds the local variance of the stochastic gradients, and \(\zeta_\star\)  is the data heterogeneity at the optimum.  The first term is the
    centralized optimal stochastic rate, the second term is a network-limited accelerated rate, and~\(\zeta_\star\) enters only logarithmically.  Consequently, the centralized statistical rate is preserved for up to
    $
        \smash{ M \lesssim \sqrt{\rho}\,N^{3/4} }
    $
    machines, improving the state-of-the-art
    $
        \smash{ M \lesssim \rho\sqrt N }
    $
    maximal scaling of the best previous bounds~\citep{eisen2025enhancing}.  This improves both the dependence on the global number of gradient computations $N$ and the dependence on the spectral gap $\rho$.
    % , and matches the lower-bound scaling up to logarithmic factors.

    \item We prove a matching lower bound for the network-dependent term.  On path networks, any method in the above first-order model needs error at least \(\Omega(LR^2M^2/(\rho N^2))\) in the deterministic regime, after translating the communication-round lower bound into the sample-budget notation of our distributed model.  Together with the classical stochastic lower bound \(\Omega(\sigma R/\sqrt N)\), this shows that both terms in our upper bound are unavoidable up to logarithmic factors.

    \item We additionally provide a guarantee for the decentralized stochastic smooth and \(\mu\)-strongly convex setting via a restarted variant of our method.  This restart scheme reaches accuracy \(\epsilon\) with total sample complexity 
    \(
        \smash{
        N
        =
        \widetilde O(
            {\sigma^2}/{\mu\epsilon}
            +
            M\sqrt{{L}/{\mu\rho}}
            \log(1/{\epsilon})
        )
        }.
    \)
    % \[
    %     N
    %     =
    %     \widetilde O\!\left(
    %         \frac{\sigma^2}{\mu\epsilon}
    %         +
    %         M\sqrt{\frac{L}{\mu\rho}}
    %         \log\frac{\Delta_0}{\epsilon}
    %     \right).
    % \]
    We defer this extension and its proof to \cref{app:strongly-convex-restarts}.

    \item A key component of our analysis is a convergence bound for accelerated stochastic optimization with one-step-delayed gradients.  The bound shows that, by using momentum-compensated query points, one can retain the usual stochastic accelerated rate even though the gradient needed to form the next standard Nesterov query is not yet available.  This delayed-acceleration result may be of independent interest, and is described in \cref{sec:one-step-delayed}.

    \item We complement the theory with synthetic least-squares experiments on ring and torus networks, illustrating the optimization performance of \dadq{} and its scaling with the number of workers relative to D-SGD and DAT-SGD (\cref{sec:experiments}).
\end{enumerate}

\subsection{Related work}

\paragraph{Decentralized stochastic optimization with gossip.}
Combining gossip averaging with local gradient computation for distributed optimization was pioneered by \citet{nedic2009distributed}, who proposed and analyzed distributed subgradient methods for multi-agent convex optimization.  \citet{lian2017can} introduced Decentralized Parallel SGD (D-PSGD), showing that decentralized training can match the linear speedup of centralized methods while avoiding the communication bottleneck of parameter servers.  \citet{assran2019stochastic} extended gossip-based training to directed and time-varying topologies via stochastic gradient push, improving robustness to stragglers and asymmetric links.

Unified convergence analyses for decentralized SGD under changing topologies were subsequently developed by \citet{koloskova2020unified}, with rates depending on the spectral gap and data heterogeneity.  The gradient-tracking approach of \citet{nedic2017achieving} and its refined analysis by \citet{koloskova2021improved} can eliminate explicit heterogeneity terms at the cost of a weaker dependence on the spectral gap.  Variance-reduction and compression techniques have also been studied~\citep{tang2018communication,tang2018d2,li2019communication}.  Most recently, \citet{eisen2025enhancing} combined anytime averaging with accelerated stepsizes to achieve a critical parallelism of $M \lesssim \rho\sqrt{N}$, the best prior scaling.  These methods, however, rely on non-accelerated stochastic approximation and therefore fall short of the $N^{3/4}$-scale critical parallelism attainable in centralized accelerated settings.

\paragraph{Classical and centralized acceleration.}
Acceleration was first developed for centralized first-order optimization, beginning with Nesterov's accelerated gradient method~\citep{nesterov1983method}, which attains the optimal deterministic rate \(O(1/T^2)\) for smooth convex optimization after \(T\) steps.  \citet{lan2012optimal} extended acceleration to stochastic composite optimization, combining this fast deterministic term with the optimal \(O(\sigma/\sqrt T)\) stochastic rate.  Using acceleration in centralized minibatching improves the number of workers that preserve the centralized statistical rate from \(M\lesssim\sqrt N\) for synchronized minibatch SGD to \(M\lesssim N^{3/4}\)~\citep{dekel2012optimal}.  The centralized inexact-oracle analysis of \citet{devolder2014first} shows, however, that accelerated methods are sensitive to persistent oracle errors.

\paragraph{Decentralized and asynchronous acceleration.}
Accelerated methods have also been developed for deterministic decentralized optimization over networks, including optimal strongly convex rates on fixed networks~\citep{scaman2017optimal}, smooth convex rates~\citep{xu2020accelerated}, and time-varying-network guarantees~\citep{rogozin2020optimal,rogozin2021towards,kovalev2021adom}. In this line of work, accelerated gossip is incorporated through blocking communication phases~\citep[e.g.,][]{scaman2017optimal,ye2023multi}.
Related asynchronous stochastic methods make acceleration compatible with delayed gradients through minibatching, at the cost of filtering out a constant fraction of the samples~\citep{tyurin2023optimal,attia2025faster}.  Our method instead uses one-step delayed acceleration to exploit every micro-round for both sampling and communication, while controlling the disagreement-induced bias.

\ifarxiv
\paragraph{Gossip algorithms and consensus.}
Gossip protocols have a long history in distributed computing, originating with the early work of \citet{tsitsiklis1986distributed} on distributed asynchronous gradient algorithms, which established that consensus-based gradient methods retain the convergence properties of their centralized counterparts under bounded communication delays.  Randomized gossip averaging was formalized by \citet{kempe2003gossip}, who showed that simple push-sum protocols compute aggregate statistics in $O(\log n)$ rounds on general networks, and by \citet{boyd2006randomized}, who characterized the averaging time of pairwise randomized gossip in terms of the second-largest eigenvalue of the gossip matrix and showed that optimal weight design reduces to a semidefinite program.  \citet{xiao2004fast} studied the fastest linear iterations for distributed averaging and demonstrated that optimally designed gossip matrices substantially outperform Laplacian-based heuristics.  This spectral-gap perspective, linking the mixing rate of the gossip matrix to the convergence of distributed algorithms, is central to the analysis of all subsequent gossip-based optimization methods, including ours.
\else
Due to space constraints, additional related work is discussed in \cref{app:additional-related-work}.
\fi

\section{Problem setup}
\label{sec:distributed-problem-setup}

We first specify the optimization problem, stochastic oracle model, and communication model used by the distributed method.
We want to minimize the finite-sum objective
\[
    f(x)=\frac1M\sum_{i=1}^M f_i(x),
\]
using $M$ decentralized workers.  We assume that the solution set is nonempty, and fix a minimizer $x^\star\in\arg\min_x f(x)$.  Each $f_i:\R^d\to\R$ is convex and $L$-smooth, i.e.,
$\norm{\nabla f_i(x)-\nabla f_i(y)}\le L\norm{x-y}$ for all $x,y\in\R^d$.

Worker $i$ has access to a local stochastic gradient oracle $G_i(x,\xi)$ satisfying
\[
    \E[G_i(x,\xi)\mid x]=\nabla f_i(x),
    \qquad
    \E\!\left[\norm{G_i(x,\xi)-\nabla f_i(x)}^2\mid x\right]\le \sigma^2.
\]
These are standard conditional unbiasedness and bounded-variance assumptions. All oracle samples are conditionally independent given past queries.

\paragraph{Bounded data heterogeneity.}
We assume the local objectives have \emph{bounded gradient heterogeneity at the optimum}: there is a constant $\zeta_\star\ge 0$ such that
\begin{equation}
\label{eq:bounded-heterogeneity}
    \frac1M\sum_{i=1}^M
    \norm{\nabla f_i(x^\star)-\nabla f(x^\star)}^2
    \le
    \zeta_\star^2.
\end{equation}
This is an optimum-local heterogeneity condition, similar to the assumption 
used by \citet{koloskova2020unified}, and should be distinguished
from the stronger global bounded-heterogeneity assumptions often used in
decentralized stochastic optimization, which require this average squared
gradient mismatch to be uniformly bounded for all \(x\)
\citep[e.g.,][]{lian2017can,tang2018communication,tang2018d2,li2019communication,he2022byzantine,eisen2025enhancing}.
Other relaxations bound heterogeneity at the origin \citep{lu2021optimal}.
Gradient-tracking methods can completely remove explicit heterogeneity terms
in some regimes~\citep{nedic2017achieving,koloskova2021improved}, at the
expense of weaker dependence on the spectral gap $\rho$. Our analysis here
does require a finite optimum-local heterogeneity parameter, but
\(\zeta_\star\) enters the final guarantee only logarithmically.

\paragraph{Decentralized communication.}
The $M$ workers form a fixed network with an undirected communication graph $\mathcal G=(\{1,\ldots,M\},E)$.  The basic time unit is a synchronous micro-round.  In one micro-round, each worker may perform local computation, make at most one stochastic-gradient oracle call, and perform a small constant number (ideally one) of neighbor communication step; there is no parameter server or all-reduce operation.  We restrict attention to nonblocking schedules: algorithms may not insert additional phases consisting of many back-to-back gossip steps with no local oracle work.  Thus a budget of $N$ stochastic-gradient computations over $M$ workers corresponds to at most $N/M$ micro-rounds, and communication must be interleaved with those rounds rather than appended in extra blocking phases.

As in the gossip-based decentralized model of \citet{scaman2017optimal}, communication is implemented through a \emph{gossip matrix} $P\in[0,1]^{M\times M}$ supported on the graph, so $P_{ij}=0$ whenever $i\ne j$ and $\{i,j\}\notin E$. One communication round with $P$ acts as follows: when the local vectors are the columns of $Z\in\R^{d\times M}$, all workers simultaneously exchange messages along the support of $P$ and compute the corresponding columns of $ZP$. 
We assume that $P$ is symmetric and doubly stochastic.  Formally,
\[
    P=P^\top,
    \qquad
    P\mathbf 1=\mathbf 1,
    \qquad
    \mathbf 1^\top P=\mathbf 1^\top.
\]
% For a matrix $Z\in\R^{d\times M}$ whose columns are the local worker variables, write
% \[
%     \bar Z:=Z\frac{\mathbf 1\mathbf 1^\top}{M}.
% \]
Let the eigenvalues of $P$ be ordered as
$
    1=\lambda_1(P)\ge \lambda_2(P)\ge \cdots\ge \lambda_M(P)\ge -1.
$
We assume that $P$ has spectral gap at least $\rho\in(0,1]$, in the sense that
\begin{equation}
\label{eq:gossip-spectral-gap}
    1-\lambda_2(P) \geq \rho > 0.
\end{equation}
% Equivalently, the non-consensus eigenvalues of $I-P$ are at least $\rho$; since $P$ is symmetric and stochastic, they are also at most $2$.  This spectral condition is the one used by the accelerated-gossip analysis below.
% 
We note that this spectral condition is strictly weaker than the one-step consensus-contraction assumption used in previous work~\citep[e.g.,][]{koloskova2020unified,koloskova2021improved,eisen2025enhancing}.  For a fixed symmetric gossip matrix, such a contraction assumption effectively requires
$
    \max\{|\lambda_2(P)|,|\lambda_M(P)|\}<1,
$
and therefore separates the most negative eigenvalue of $P$ away from $-1$.  In contrast, our analysis only needs the weaker condition in \cref{eq:gossip-spectral-gap} and allows $\lambda_M(P)$ to be arbitrarily close to $-1$.%
\footnote{The distinction is already visible on the two-node graph with
$
    P=
    \begin{psmallmatrix}
        0 & 1\\
        1 & 0
    \end{psmallmatrix}.
$
Here \(P\) is symmetric and doubly stochastic, \(\lambda_2(P)=-1\), and therefore \(1-\lambda_2(P)=2\), so \cref{eq:gossip-spectral-gap} holds for every \(\rho\le 1\); however, for the row vector \(z=(1,-1)\), we have \(zJ=0\) and \(zP=(-1,1)\), so \(\|zP-zJ\|_2=\|z-zJ\|_2\).  Thus no strict one-step consensus contraction holds.}

\section{Algorithm}
\label{sec:distributed-algorithm}

Our method combines Nesterov-type stochastic acceleration, which provides the $O(1/T^2)$ optimization term needed for $N^{3/4}$ critical parallelism, with accelerated gossip to reduce network disagreement. The full procedure
is given in \cref{alg:distributed-delayed-gossip}, and its convergence guarantee
is stated in \cref{thm:distributed-gossip}.  We first describe the macro-step
structure conceptually, before giving the full pseudocode.

The method runs in macro steps of \(B\) synchronous micro-rounds.  At the
beginning of macro step \(t\), the algorithm has three states: a primal state
\(X_t\), a dual/gradient-accumulation state \(Z_t\), and a query state
\(Q_t\).  Over \(T\) macro steps and \(M\) workers, the method uses
\(N=MBT\) stochastic-gradient samples.  Conceptually, each macro step pipelines
two activities: workers sample at the already prepared delayed query \(Q_t\),
while they gossip the current primal state \(X_t\) so it can define the next
query.  The important departure from standard
Nesterov-type stochastic acceleration is that \(Q_t\) is not formed from the
current primal state.  It is a one-step-delayed query, computed from previously
mixed primal copies.
\cref{sec:one-step-delayed} analyzes this technique in the classical
single-node setting.

This delay lets one macro step use its \(B\) micro-rounds for two purposes at
once.  First, every worker samples \(B\) stochastic gradients at the already
available query point \(q_{t,i}\), forming a local minibatch and using the full
sampling budget.  Second, during the same \(B\) micro-rounds, the workers run
accelerated gossip on the current primal copies \(X_t\).  The gossip block is
Nesterov acceleration applied to the quadratic disagreement objective associated
with \(I-P\); see \cref{app:nesterov-gossip}.  Thus the point being sampled is
fixed and available, while the point being communicated is prepared for the next
query.  By the end of the macro step, gossip has produced a nearly synchronized
primal copy \(\widetilde X_t\), which is then used to form \(Q_{t+1}\).
We emphasize that $Q_{t+1}$ does not depend on the current minibatch gradient; the subsequent state update is local and requires no extra communication or synchronization.

The one-step delay is what makes this separation possible: sampling does not
wait for communication, and communication does not disturb the point being
queried.  The price is that the query sequence must still behave like a
Nesterov query sequence.  For this reason \(Q_{t+1}\) includes a momentum
compensation term, which anticipates the next inertial displacement and makes
the delayed query mismatch controllable by the same stability terms that drive
the accelerated analysis.

\begin{algorithm}[ht]
\caption{Decentralized Doubly Accelerated SGD (\dadq)}
\label{alg:distributed-delayed-gossip}
\begin{algorithmic}[1]
\STATE {\bf Input:} common $x_0\in\R^d$, macro horizon $T$, local minibatch/accelerated-gossip length $B$, symmetric doubly stochastic gossip matrix $P$, gossip stepsize $\alpha_{\rm g}\in(0,1]$, gossip momentum $\beta_{\rm g}\in[0,1)$, weights $(a_t)_{t=0}^{T-1}$, averaging parameters $(\theta_t)_{t=0}^T$, momentum parameters $(m_t)_{t=0}^{T+1}$, and dual stepsizes $(\eta_t)_{t=0}^T$.
\STATE Set $X_{-1}=X_0=Z_0=Q_0=\widetilde X_{-1}=x_0\mathbf 1^\top$.
\FOR{$t=0,1,\ldots,T$}
    \STATE Initialize $Y_t^{(-1)}=Y_t^{(0)}=X_t$ and $g_{t,i}=0$ at each node $i$.
    \FOR{$s=1,2,\ldots,B$}
        \STATE Each node $i$ samples one gradient at its current delayed query $q_{t,i}$, the $i$th column of $Q_t$, and accumulates
        \[
            g_{t,i}\leftarrow g_{t,i}+\frac1B G_i(q_{t,i},\xi_{t,i}^{(s)}).
        \]
        \STATE In the same round, perform one Nesterov accelerated-gossip step:\footnotemark
        \[
            V_t^{(s)}
            =
            Y_t^{(s-1)}+\beta_{\rm g}\bigl(Y_t^{(s-1)}-Y_t^{(s-2)}\bigr),
            \qquad
            Y_t^{(s)}
            =
            (1-\alpha_{\rm g})V_t^{(s)}+\alpha_{\rm g}V_t^{(s)}P .
        \]
    \ENDFOR
    \STATE Set the accelerated mixed primal copy
    \(
        \widetilde X_t
        =
        Y_t^{(B)},
    \)
    and form the next delayed query matrix
    \[
        Q_{t+1}
        =
        \widetilde X_t
        +
        m_t(1+m_{t+1})(\widetilde X_t-\widetilde X_{t-1}),
    \]
    whose columns are $q_{t+1,i}$.
    \STATE Set $G_t=[g_{t,1},\ldots,g_{t,M}]$ and update local dual and primal copies:
    \[
        Z_{t+1}=Z_t-\eta_tG_t,
        \qquad
        X_{t+1}=(1-\theta_t)X_t+\theta_t Z_{t+1}.
    \]
\ENDFOR
\STATE Each node $i$ outputs its local column $\widetilde x_{T,i}$ of $\widetilde X_T$.
\end{algorithmic}
\end{algorithm}
\footnotetext{Only multiplication by $P$ requires neighbor communication.  Scalar multiplications and additions/subtractions are local, so each recurrence step uses one gossip communication round.}

Previous decentralized methods \citep{koloskova2020unified,eisen2025enhancing} tightly couple the variables used for communication with the variables used for gradient computation.  Although such coupling can be useful, our one-step-delayed construction deliberately separates the two roles.  This separation is what allows us to combine accelerated gossip with accelerated optimization, two tools that are powerful but sensitive to inexactness~\citep{devolder2014first}. An alternative way to accommodate these techniques is to reserve some rounds purely for communication, effectively leaving a constant fraction of the available gradient queries unused, as in the approaches of \citet{tyurin2023optimal,attia2025faster}. Such blocking communication phases are outside the nonblocking round model above; the one-step-delayed scheme avoids them because every micro-round can be used both for sampling a stochastic gradient and for performing a gossip step.

\section{Main results}
\label{sec:distributed-main-result}

Next we present our convergence guarantee for \cref{alg:distributed-delayed-gossip}, and the matching lower bound for linear-span decentralized methods.

\begin{theorem}[Distributed delayed-gossip rate]
\label{thm:distributed-gossip}
Let each $f_i$, $i=1,\ldots,M$, be convex and $L$-smooth, with an unbiased oracle $G_i$ of conditional variance at most $\sigma^2$. Assume conditionally independent oracle samples and optimum-local heterogeneity bounded by $\zeta_\star$, and let $P$ be a fixed symmetric doubly stochastic gossip matrix with spectral gap at least $\rho\in(0,1]$. Let $f\coloneqq \frac1M\sum_{i=1}^M f_i$, $x^\star\in\arg\min_x f(x)$, and $R>0$ satisfy
$
    R\ge \norm{x_0-x^\star}.
$
Fix a total sample budget $N$, and define
\[
    \Lambda_N
    :=
    \frac{100 N}{\sqrt M}
    \left(
        1+
        \frac{\zeta_\star+\sigma}{LR}
    \right).
\]
Choose
$
    B
    :=
    \bigl\lceil
        \max\{
            1,\,
            (20/\sqrt\rho) \log(\Lambda_N/\rho)
        \}
    \bigr\rceil,
$
and assume for simplicity that $T:=N/(MB)$ is a positive integer.
Set the accelerated gossip parameters
$\alpha_{\rm g}:=\tfrac12$,\;
$\beta_{\rm g}:=(1-\sqrt{\rho/2})/(1+\sqrt{\rho/2})$,
and the algorithmic parameters
\begin{gather*}
    a_t:=\frac{t+1}{128L},\quad
    A_t:=\sum_{i=0}^t a_i=\frac{(t+1)(t+2)}{256L},\quad
    \theta_t:=\frac{a_t}{A_t}=\frac{2}{t+2}\;\;(0\le t<T),\quad
    \theta_T:=0, \\
    m_0:=0,\quad
    m_t:=\frac{t-1}{t+2}\;\;(1\le t\le T),\quad
    m_{T+1}:=0, \\
    H_{T-1}:=\sqrt{\textstyle\sum_{i=0}^{T-1}a_i^2},\quad
    \beta:=1+\frac{\sigma}{R\sqrt{MB}}H_{T-1},\quad
    \eta_t:=\frac{a_t}{\beta}\;\;(0\le t<T),\quad
    \eta_T:=0.
\end{gather*}
Run \cref{alg:distributed-delayed-gossip} with these parameters for $T$ macro steps, and let $\widetilde x_{T,i}$ be the output held by node $i$.  Then, for every node $i$,
\[
    \E[f(\widetilde x_{T,i})-f(x^\star)]
    \le
    \frac{14\sigma R}{\sqrt N}
    +
    \frac{2^{17}LR^2M^2}{\rho N^2}
    \log^2 \frac{\Lambda_N}{\rho}.
\]
\end{theorem}

The optimization parameters adapt a standard Nesterov stochastic-acceleration schedule to the one-step-delayed, momentum-compensated query. The gossip parameters $\alpha_{\rm g},\beta_{\rm g}$ are standard Nesterov choices for the consensus quadratic. The block length $B$ makes residual disagreement small and also acts as a minibatch size, bounding the averaged conditional oracle variance by $\sigma^2/(MB)$.

We give a proof sketch in \cref{sec:distributed-gossip-key-ideas} and defer the full proof of \cref{thm:distributed-gossip} to \cref{app:distributed-delayed-gossip}.
Next, we state a lower bound for linear-span methods, showing the guarantee is nearly optimal.

\begin{theorem}[Lower bound for linear-span methods]
\label{thm:linear-span-lower-bound}
There is a universal constant \(c>0\) such that the following holds.  For
every \(M\ge6\), \(L>0\), \(R>0\), and \(S\ge1\), on the \(M\)-node path graph
with spectral gap \(\rho=\Theta(M^{-2})\), there exist convex \(L\)-smooth
functions \(f_1,\ldots,f_M:\R^d\to\R\) in dimension
\(d=\Theta(1+S\sqrt\rho)\) such that, for
$
    f(x)=\frac1M\sum_{i=1}^M f_i(x),
$
the minimizer \(x^\star\) satisfies \(\norm{x^\star}\le R\), and every
deterministic linear-span decentralized first-order method run for \(S\) gossip
communication rounds has averaged output \(\bar x_S\) satisfying
\[
    f(\bar x_S)-f(x^\star)
    \ge
    cLR^2
    \min\left\{1,\frac1{\rho S^2}\right\}.
\]
\end{theorem}

The proof of \cref{thm:linear-span-lower-bound} appears in
\cref{app:linear-span-lower-bound}.  The theorem is a communication lower
bound: on the path, the hard instance forces useful coordinate information to
alternate between two blocks separated by graph distance
\(\Theta(\rho^{-1/2})\), so each such alternation costs
\(\Theta(\rho^{-1/2})\) gossip rounds.  In our distributed model over $M$ machines, 
a budget of \(N\) local gradients implies \(\Theta(N/M)\) gossip rounds, since each 
of the \(M\) nodes takes one sample and performs one gossip step per distributed
round.  Substituting \(S=\Theta(N/M)\) in the asymptotic regime
\(S\gtrsim \rho^{-1/2}\) gives
\[
    \Omega\!\left(\frac{LR^2M^2}{\rho N^2}\right).
\]
Thus, in the deterministic case, the network-dependent term in
\cref{thm:distributed-gossip} is unimprovable up to logarithmic factors for
deterministic linear-span decentralized methods, even on path networks.  The
stochastic term is optimal for an even simpler reason: even if communication were
free, the problem would contain ordinary stochastic convex optimization with
\(N\) stochastic-gradient samples, whose minimax error is
\(\smash{\Omega(\sigma R/\sqrt N)}\) \citep[e.g.,][]{agarwal2012information}.  Thus,
the upper bound in \cref{thm:distributed-gossip} is optimal up to logarithmic 
factors (at least for linear-span decentralized methods whose only randomness comes from 
the stochastic gradients): the first term is forced by stochastic
oracle noise, and the second term is forced by network-limited information
propagation.

\subsection{Proof sketch of \texorpdfstring{\cref{thm:distributed-gossip}}{Theorem 1}}
\label{sec:distributed-gossip-key-ideas}

\cref{alg:distributed-delayed-gossip} combines two Nesterov-type acceleration mechanisms: accelerated optimization of the iterates and accelerated gossip for reducing network disagreement.  This combination is delicate because accelerated gradient methods are notoriously sensitive to persistent oracle errors: even small bias or inexactness in the gradient information can accumulate severely across iterations~\citep{devolder2014first}. In the decentralized setting such inexactness arises naturally, because limited communication means that the local gradients are evaluated at different local query points rather than at the exact network average.  To use acceleration without losing its rate, the disagreement between the local query points must therefore be sufficiently small.  This is incompatible with querying a newly formed accelerated point immediately: the point must first be mixed sufficiently across the network. The one-step-delayed accelerated method resolves this tension: it lets the algorithm sample gradients at query points that were already computed and communicated in previous rounds, while the current primal copy undergoes accelerated gossip to within sufficient mixing error for use in the next query.

\paragraph{One-step-delayed acceleration.}

The proof is built on a primal-dual recursion $z_{t+1}=z_t-(a_t/\beta)g_t$, $x_{t+1}=(1-\theta_t)x_t+\theta_t z_{t+1}$, initialized at $z_0=x_0$, with weights $a_t>0$, averaging parameters $\theta_t=a_t/A_t$ where $A_t=\sum_{i=0}^t a_i$, and regularization $\beta>0$.  By a standard telescoping argument (substituting $a_tg_t=\beta(z_t-z_{t+1})$ and using the three-point identity), combined with smoothness and convexity of $f$ at a query point $r_t$, one obtains for any comparator $u$ (later set to a minimizer $x^\star$) the estimate
\begin{equation}\label{eq:master-estimate}
    A_{T-1}\bigl(f(x_T)-f(u)\bigr)
    \le
    \frac{\beta}{2}\norm{u-x_0}^2
    -
    \frac{\beta}{2}\sum_{t=0}^{T-1}\norm{z_{t+1}-z_t}^2
    +
    \frac L2\sum_{t=0}^{T-1}A_t\norm{x_{t+1}-r_t}^2 .
\end{equation}
The first term is a one-time distance penalty.  The second, negative term, is a ``stability budget'' that can absorb positive errors; its presence is what makes acceleration work.  The challenge is to choose the query point $r_t$ so that the last sum is absorbed by this stability budget.

\emph{Standard acceleration (no delay).}
In the standard Nesterov scheme one queries at the momentum point
\[
    y_t:=(1-\theta_t)x_t+\theta_t z_t
    =
    x_t+m_t(x_t-x_{t-1}),
\]
where \(m_t\) is the corresponding momentum coefficient.  Setting \(r_t=y_t\) in \cref{eq:master-estimate}, the distance from \(x_{t+1}\) to the query point simplifies to \(x_{t+1}-y_t=\theta_t(z_{t+1}-z_t)\), so
\(
    \frac{LA_t}{2}\norm{x_{t+1}-y_t}^2
    =
    \frac{LA_t\theta_t^2}{2}\norm{z_{t+1}-z_t}^2.
\)
For the parameter choice \(a_t=(t+1)/(128L)\), the coefficient \(LA_t\theta_t^2/2\) is at most \(\beta/2\), so each term in the positive sum is dominated by the corresponding term in the negative stability sum.  This cancellation is what gives the accelerated \(O(1/T^2)\) rate.

\emph{The delay problem and the compensated query.}
With a one-step delay, the gradient used at time \(t\) must have been queried \emph{before} the current iterate \(x_t\) is known (because the gradient oracle returns its result one step late).  Therefore, we cannot query at \(y_t\), which depends on \(x_t\).  The naive fallback of querying at the stale momentum point \(y_{t-1}\) introduces a mismatch \(y_t - y_{t-1}\) that contains a large inertial displacement proportional to \(x_t - x_{t-1}\); this displacement is \emph{not} proportional to \(\theta_t(z_{t+1}-z_t)\) and hence cannot be absorbed by the stability budget.

Our key observation is that a \emph{momentum-compensated} query point, defined for \(t\ge 1\) by
\[
    q_t:=x_{t-1}+m_{t-1}(1+m_t)(x_{t-1}-x_{t-2}),
    \qquad q_0:=x_0,
\]
resolves this problem.  The extra factor \((1+m_t)\) anticipates the inertial displacement that will appear when the next momentum point \(y_t\) is formed.  Concretely, it is chosen so that the identity
\[
    y_t-q_t=(1+m_t)(x_t-y_{t-1})
\]
holds, which rewrites the delay mismatch entirely in terms of the iterate displacement \(x_t - y_{t-1}\).  Since \(x_t - y_{t-1} = \theta_{t-1}(z_t - z_{t-1})\) by the primal-dual recursion, this displacement \emph{is} controlled by the stability budget.  Quantitatively, with the parameters of \cref{thm:main} (the one-step delayed accelerated rate proven in \cref{sec:one-step-delayed}),
\[
    \sum_{t=1}^{T-1} A_t\norm{y_t-q_t}^2
    \le
    12\sum_{t=0}^{T-2} A_t\norm{x_{t+1}-y_t}^2 .
\]
Since \(x_{t+1}-q_t=(x_{t+1}-y_t)+(y_t-q_t)\), the total cost with query point \(r_t=q_t\) in \cref{eq:master-estimate} is bounded by a constant multiple of \(\sum_t A_t\theta_t^2\norm{z_{t+1}-z_t}^2\), which is still absorbed by the negative stability term.

\emph{Restoring acceleration:} the compensated query makes the delayed scheme pay essentially the same smoothness cost as the non-delayed one, so the accelerated \(O(1/T^2)\) rate is preserved despite the one-step delay.  The full stochastic and biased-oracle analysis in \cref{thm:main} starts from the same estimate~\cref{eq:master-estimate} and adds error terms for stochastic noise \(g_t - \nabla f(q_t)\) and residual disagreement bias.

\paragraph{Accelerated gossip.}

The second mechanism is accelerated gossip, which reduces inter-node disagreement exponentially fast while preserving the network average.  The key observation is that, per coordinate, a single (lazy) gossip round $v \mapsto (1-\alpha_{\rm g})v + \alpha_{\rm g}\,vP$ is exactly a gradient step with stepsize $\alpha_{\rm g}$ on the quadratic disagreement objective $F(x) := \tfrac12 x(I-P)x^\top$, whose minimizers are the consensus vectors.  On the disagreement subspace $\{x : xJ = 0\}$ (where $J := \mathbf{1}\mathbf{1}^\top/M$), $F$ is $\rho$-strongly convex and $2$-smooth, so applying Nesterov's accelerated gradient method yields exponential contraction at the optimal rate for this condition number.  After $B$ accelerated gossip rounds the Frobenius-norm disagreement of the $d\times M$ network state $Y$ contracts as
\[
    \bigl\| Y^{(B)} - Y^{(B)}J \bigr\|_F
    \le
    \frac{2}{\sqrt\rho}
    \left(1-\sqrt{\rho/2}\right)^{B/2}
    \bigl\| Y^{(0)} - Y^{(0)}J \bigr\|_F,
\]
since right-multiplication by $P$ acts on each row independently and preserves row averages.  Solving for the number of rounds needed to reach disagreement $\varepsilon$ gives the communication complexity $B = O(\rho^{-1/2}\log(\|Y^{(0)} - Y^{(0)}J\|_F/(\varepsilon\sqrt\rho)))$.  This is the $\rho^{-1/2}$ scaling of accelerated gossip, compared with $\rho^{-1}$ for ordinary gossip; see \cref{app:nesterov-gossip} for the full derivation.  In \cref{thm:distributed-gossip}, $B$ is set to make the residual disagreement small enough for the delayed accelerated estimate to hold.

\paragraph{Putting the pieces together.}

We now combine the one-step-delayed accelerated estimate and the accelerated gossip contraction to sketch how the two-term bound in \cref{thm:distributed-gossip} arises.

\emph{Dual use of each macro-round.}
Each of the \(T\) macro steps uses the same block of \(B\) micro-rounds for two concurrent purposes: (i)~every worker evaluates \(B\) stochastic gradients at the already available delayed query point, forming a local minibatch of size~\(B\); and (ii)~accelerated gossip runs for \(B\) communication rounds on the current primal copies.  The total sample budget is therefore \(N = MBT\).

\emph{Reduction to the single-node problem.}
Since accelerated gossip preserves network averages exactly, the averaged iterates \(\bar x_t = (1/M)\sum_i x_{t,i}\) follow the one-step-delayed accelerated recursion from \cref{eq:master-estimate}, driven by the averaged oracle
\(
    \widehat G_t = \frac{1}{M}\sum_{i=1}^{M} g_{t,i}.
\)
Minibatching over the \(B\) gradient samples in each macro-round reduces the conditional variance of \(\widehat G_t\) to at most \(\sigma^2/(MB)\).  The only remaining difference from a centralized gradient query at the network-average point \(\bar q_t\) is a conditional bias caused by the nodes querying at slightly different local points \(q_{t,i}\).  Let \(\calF_t\) be the history at the start of macro step \(t\), including \(Q_t\).  By \(L\)-smoothness of each \(f_i\), this bias is bounded by
\[
    \bigl\|
        \E[\widehat G_t\mid\calF_t] - \nabla f(\bar q_t)
    \bigr\|
    =
    \biggl\|
        \frac{1}{M}\sum_{i=1}^{M}
        \bigl(\nabla f_i(q_{t,i}) - \nabla f_i(\bar q_t)\bigr)
    \biggr\|
    \le
    \frac{L}{\sqrt{M}}\norm{Q_t - \bar Q_t}_F,
\]
where \(Q_t - \bar Q_t\) measures the disagreement between the local query points and their average.  The proof therefore controls a feedback loop: query disagreement creates oracle bias, while local gradient discrepancies create future disagreement.  The full proof tracks these error terms jointly over the macro steps.  The accelerated gossip length \(B\) is chosen large enough that this feedback is contractive, so the accumulated network-induced bias fits within the inexact-oracle tolerance of \cref{thm:main}.

\emph{Deriving the convergence bound.}
Applying \cref{thm:main} to the averaged dynamics yields two terms.  The \emph{stochastic term} comes from the averaged-oracle variance \(\sigma^2/(MB)\); after \(T\) macro steps it gives
\(
    \frac{\sigma R}{\sqrt{T}} \cdot \frac{1}{\sqrt{MB}}
    =
    \frac{\sigma R}{\sqrt{MBT}}
    =
    \frac{\sigma R}{\sqrt{N}},
\)
recovering the centralized statistical rate.  The \emph{deterministic (network) term} comes from the accelerated \(O(1/T^2)\) optimization rate.  Since \(T = N/(MB)\) and \(B = \widetilde\Theta(\rho^{-1/2})\), this term becomes
\(
    \frac{LR^2}{T^2}
    =
    \frac{LR^2 M^2 B^2}{N^2}
    =
    \widetilde O\!\left(\frac{LR^2 M^2}{\rho\, N^2}\right).%,
\)
% where the \(\widetilde O\) hides the \(\log^2(\Lambda_N/\rho)\) factor absorbed into \(B^2\).

\emph{Per-node guarantee.}
The above bound holds for the network-average iterate \(\bar x_T\); the per-node bound follows from the final accelerated-gossip block and smoothness.

\paragraph{What this sketch omits.}
The full proof in \cref{app:distributed-delayed-gossip} makes the following aspects precise, which are only outlined here: (i)~the stochastic error analysis, including the exact dependence on the minibatch variance and the interaction between stochastic noise and the delayed query; (ii)~the bias-budget calculation showing that the gossip contraction makes the deterministic bias terms in \cref{thm:main} sum to the claimed network-dependent rate; (iii)~the role of the heterogeneity parameter \(\zeta_\star\), which enters only through the logarithmic factor \(\Lambda_N\); and (iv)~the precise numerical constants.

\paragraph{Discussion and limitations.}
\label{par:discussion-limitations}
\cref{thm:distributed-gossip} gives a rate that is optimal up to logarithmic factors within the stated linear-span class for decentralized stochastic smooth convex optimization in our synchronous, nonblocking gossip model on fixed undirected networks. In particular, its network-dependent term matches our lower bound in \cref{thm:linear-span-lower-bound} on path networks up to these factors. These comparisons apply within the nonblocking model; other communication schedules may have different rates and tradeoffs.
Parameter choices in \cref{thm:distributed-gossip} require knowledge of problem and network quantities, or valid bounds on them. Estimating these quantities and tuning the method can be difficult in practice, and attaining the same rate without such prior knowledge remains an open question. In addition, logarithmic factors and conservative constants may limit the gains under limited budgets.

\section{Experimental evaluation}
\label{sec:experiments}

We evaluate the optimization performance of \dadq{} and its scaling with the number of workers by comparing it with D-SGD~\citep{koloskova2020unified} and DAT-SGD~\citep{eisen2025enhancing} on a synthetic least-squares problem following the setup of \citet{eisen2025enhancing}.  Two variants of \dadq{} are considered: one using the momentum schedule of \cref{thm:distributed-gossip}, and a practical variant with constant momentum $m_t=\beta_{\rm g}=0.9$, $\theta_t=0.1$, $\alpha_{\rm g}=\frac12$ and $\eta_t=10\eta$, where $\eta$ is the tuned stepsize.  The experiments use ring and torus networks, whose gossip matrices are fixed and symmetric.
% , as assumed in our analysis. TK: we don't assume symmetry right (?)

\paragraph{Experimental setup.}
We consider networks of $M\in\{4,9,25,49,100\}$ workers.  Worker $i$ has local objective $f_i(x)\!=\!\frac12\norm{A_ix-b_i}^2$, with Gaussian data of dimension $d=50$.  We vary the gradient noise variance $\sigma^2$ and heterogeneity $\zeta^2$ over $\{1,10\}$.  Every method runs for $10^5$ communication rounds, using one stochastic gradient per worker per round.
For each method and configuration, we tune the stepsize to minimize final error, jointly with $B$ for \dadq{}. Results are averaged over three random seeds.
Full details, results for all configurations, and block-length sensitivity appear in \cref{app:experiments}.

% Trial layout: complete table beside two stacked plot placeholders.
\begin{center}
\begin{minipage}[t]{0.64\linewidth}
\vspace{0pt}
\centering
\small
\setlength{\tabcolsep}{2.5pt}
\renewcommand{\arraystretch}{1.28}
\begin{tabular*}{\linewidth}{@{\extracolsep{\fill}}llrrrr@{}}
\toprule
& & & & \multicolumn{2}{c}{\dadq{}} \\
\cmidrule(lr){5-6}
Graph & $M$ & D-SGD & DAT-SGD & constant & theorem \\
\midrule
Ring & 4 & $4.8\cdot 10^{-6}$ & $1.3\cdot 10^{-5}$ & $\mathbf{3.1\cdot 10^{-8}}$ & $5.0\cdot 10^{-7}$ \\
 & 9 & $3.0\cdot 10^{-4}$ & $6.4\cdot 10^{-6}$ & $\mathbf{4.7\cdot 10^{-9}}$ & $9.1\cdot 10^{-8}$ \\
 & 25 & $6.1\cdot 10^{-3}$ & $6.8\cdot 10^{-6}$ & $\mathbf{1.0\cdot 10^{-9}}$ & $2.8\cdot 10^{-8}$ \\
 & 49 & $1.9\cdot 10^{-2}$ & $3.0\cdot 10^{-5}$ & $\mathbf{6.3\cdot 10^{-10}}$ & $1.4\cdot 10^{-8}$ \\
 & 100 & $4.3\cdot 10^{-2}$ & $6.7\cdot 10^{-4}$ & $9.2\cdot 10^{-5}$ & $\mathbf{6.0\cdot 10^{-7}}$ \\
\midrule
Torus & 4 & $5.8\cdot 10^{-6}$ & $1.3\cdot 10^{-5}$ & $\mathbf{3.1\cdot 10^{-8}}$ & $5.0\cdot 10^{-7}$ \\
 & 9 & $9.4\cdot 10^{-6}$ & $7.0\cdot 10^{-6}$ & $\mathbf{4.7\cdot 10^{-9}}$ & $9.1\cdot 10^{-8}$ \\
 & 25 & $5.6\cdot 10^{-5}$ & $2.7\cdot 10^{-6}$ & $\mathbf{1.0\cdot 10^{-9}}$ & $2.8\cdot 10^{-8}$ \\
 & 49 & $1.4\cdot 10^{-4}$ & $1.7\cdot 10^{-6}$ & $\mathbf{6.3\cdot 10^{-10}}$ & $1.4\cdot 10^{-8}$ \\
 & 100 & $3.7\cdot 10^{-4}$ & $9.1\cdot 10^{-7}$ & $\mathbf{3.0\cdot 10^{-10}}$ & $8.1\cdot 10^{-9}$ \\
\bottomrule
\end{tabular*}
\captionof{table}{Final error $\frac1M\sum_{i=1}^M\norm{x_i-x^\star}^2$ on ring and torus networks for $\sigma^2=\zeta^2=1$. Both \dadq{} momentum schedules are shown.}
\label{tab:comparison}

\end{minipage}\hfill
\begin{minipage}[t]{0.34\linewidth}
\vspace{0pt}
\centering
\includegraphics[width=\linewidth,trim=4bp 6bp 4bp 6bp,clip]{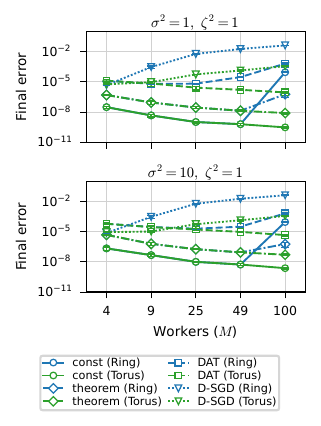}
\captionsetup{justification=raggedright,singlelinecheck=false}
\captionof{figure}{Final error versus the number of workers ($\zeta^2=1$).}
\label{fig:comparison}
\end{minipage}
\end{center}

\paragraph{Results.}
\dadq{} achieves lower final error than both baselines across all tested configurations (\cref{tab:comparison}).  The constant and theorem momentum schedules improve on DAT-SGD by at least $6$-fold and $12$-fold, respectively.  At a fixed per-worker budget, \dadq{} benefits from additional workers through $M=100$ on the torus and $M=49$ on the ring (\cref{fig:comparison}), whereas DAT-SGD's ring error is lowest at $M=9$ or $25$.  On the $100$-worker ring with $\sigma^2=\zeta^2=1$, the theorem schedule substantially reduces the error relative to constant momenta: $6.0\cdot10^{-7}$ versus $9.2\cdot10^{-5}$.

\ifarxiv\else\clearpage\fi
\section*{Acknowledgements}

This project has received funding from the European Research Council (ERC) under the European
Union's Horizon 2020 research and innovation program (grant agreement No.\ 101078075). Views and
opinions expressed are however those of the author(s) only and do not necessarily reflect those of the
European Union or the European Research Council. Neither the European Union nor the granting
authority can be held responsible for them. This work received additional support from the Israel
Science Foundation (ISF, grant numbers 2549/19 and 3174/23), from the Council for Higher Education in Israel under a Moonshot Project, a grant from the Tel Aviv University Center for AI and Data Science (TAD), from the Len Blavatnik and the Blavatnik Family foundation, and a fellowship from the
Israeli Council for Higher Education.

\bibliographystyle{plainnat}
\bibliography{references}

\clearpage
\appendix
\crefalias{section}{appendix}

\ifarxiv\else
\section{Additional related work}
\label{app:additional-related-work}

\paragraph{Gossip algorithms and consensus.}
Gossip protocols have a long history in distributed computing, originating with the early work of \citet{tsitsiklis1986distributed} on distributed asynchronous gradient algorithms, which established that consensus-based gradient methods retain the convergence properties of their centralized counterparts under bounded communication delays.  Randomized gossip averaging was formalized by \citet{kempe2003gossip}, who showed that simple push-sum protocols compute aggregate statistics in $O(\log n)$ rounds on general networks, and by \citet{boyd2006randomized}, who characterized the averaging time of pairwise randomized gossip in terms of the second-largest eigenvalue of the gossip matrix and showed that optimal weight design reduces to a semidefinite program.  \citet{xiao2004fast} studied the fastest linear iterations for distributed averaging and demonstrated that optimally designed gossip matrices substantially outperform Laplacian-based heuristics.  This spectral-gap perspective, linking the mixing rate of the gossip matrix to the convergence of distributed algorithms, is central to the analysis of all subsequent gossip-based optimization methods, including ours.
\fi

\section{One-step-delayed accelerated algorithm}
\label{sec:one-step-delayed}

In this section we isolate the one-step-delayed acceleration component of our gossip method.  Through minibatching and accelerated gossip, the network average reduces to a single-node stochastic optimization problem with a delayed oracle, while residual disagreement enters as a controlled oracle bias.  We present an accelerated algorithm for this reduced problem and prove that momentum-compensated queries retain the usual stochastic accelerated rate.

\subsection{Problem setup}

We consider the convex stochastic optimization problem
\[
    \min_{x\in\R^d} f(x),
\]
where $f:\R^d\to\R$ is convex and differentiable.  We assume that $f$ has an $L$-Lipschitz gradient:
\[
    \norm{\nabla f(x)-\nabla f(y)}\le L\norm{x-y},
    \qquad x,y\in\R^d.
\]
Equivalently,
\[
    f(u)
    \le
    f(v)+\ip{\nabla f(v)}{u-v}+\frac L2\norm{u-v}^2.
\]
Let $x^\star\in\arg\min_{x\in\R^d} f(x)$, and assume that $R>0$ satisfies
\[
    R\ge \norm{x_0-x^\star}.
\]
If the minimizer is not unique, one may take $x^\star$ to be any minimizer and $R$ to be any known upper bound on its distance from $x_0$; the sharpest such choice is the distance from $x_0$ to the solution set.

\paragraph{Delayed stochastic oracle.}

At round $t$, the algorithm receives a stochastic gradient queried at the previous round.  If the point queried at the previous round is $q_t$, then the returned random vector is
\[
    G_t = G(q_t,\xi_t)=\nabla f(q_t)+\zeta_t.
\]
The stochastic error satisfies the RMS conditional-bias and centered-variance assumptions
\begin{equation}
\label{eq:noise-assumption}
    b_t:=\E[\zeta_t\mid \calF_t],
    \qquad
    \E\norm{b_t}^2 \le \epsilon_t^2,
    \qquad
    \E[\norm{\zeta_t-b_t}^2\mid \calF_t]\le \sigma^2,
\end{equation}
where $\epsilon_0,\ldots,\epsilon_{T-1}$ are deterministic nonnegative RMS bias bounds for the horizon under consideration, and $\calF_t$ contains all randomness and all iterates available before $G_t$ is revealed.  In particular, $q_t$ is $\calF_t$-measurable.  The unbiased model is recovered by taking $\epsilon_t=0$ for all $t$.
In the unconstrained setting we also assume that the oracle calls made by the algorithm are square-integrable:
\begin{equation}
\label{eq:square-integrable}
    \E\norm{G(q_t,\xi_t)}^2<\infty,
    \qquad 0\le t<T.
\end{equation}

\subsection{Algorithm}

The algorithm is written for prescribed weights $a_t$, averaging parameters $\theta_t$, momentum parameters $m_t$, and prox coefficient $\beta$.  The concrete parameter values used in the convergence theorem are stated in \cref{thm:main}.

At round $t$, the next query point $q_{t+1}$ is already computable from the pre-gradient state $(x_t,x_{t-1})$ and the known momentum coefficients; it does not depend on the delayed gradient $G_t$.  This is what makes the method compatible with a strict one-step delay.  The method then receives the stochastic gradient of the point submitted one round earlier, takes the dual gradient step $z_{t+1}=z_t-(a_t/\beta)G_t$, and forms the accelerated primal iterate $x_{t+1}$ by averaging $x_t$ with $z_{t+1}$.  This recursion is equivalent to the cumulative prox/FTRL (Follow-the-Regularized-Leader) update $z_{t+1}=x_0-\beta^{-1}\sum_{i=0}^t a_iG_i$, minimizing cumulative linear losses plus a quadratic regularizer.

\begin{algorithm}[ht]
\caption{One-step delayed momentum-compensated accelerated stochastic approximation}
\label{alg:delayed-comp-acsa}
\begin{algorithmic}[1]
\STATE Input: $x_0\in\R^d$, horizon $T$, parameters $(a_t,\theta_t)_{t=0}^{T-1}$, $(m_t)_{t=0}^T$, and $\beta>0$.
\STATE Set $z_0=x_0$, $x_{-1}=x_0$, and $q_0=x_0$.
\STATE Submit the initial query at $q_0$.
\FOR{$t=0,1,\ldots,T-1$}
    \STATE Compute the next compensated query point
    \[
        q_{t+1}
        =
        x_t+m_t(1+m_{t+1})(x_t-x_{t-1}),
    \]
    with $m_0=0$.
    \STATE Submit the next-round query at $q_{t+1}$; receive the delayed stochastic gradient $G_t=G(q_t,\xi_t)$ from the previously submitted query at $q_t$.
    \STATE Take the dual gradient step
    \[
        z_{t+1}=z_t-\frac{a_t}{\beta}G_t.
    \]
    \STATE Compute the accelerated primal point
    \[
        x_{t+1}=(1-\theta_t)x_t+\theta_t z_{t+1}.
    \]
\ENDFOR
\STATE Output $x_T$.
\end{algorithmic}
\end{algorithm}

\paragraph{Compensated query points.}
For a Nesterov-type momentum sequence, the extrapolated point is
\[
    y_t=(1-\theta_t)x_t+\theta_t z_t
        = x_t+m_t(x_t-x_{t-1}).
\]
If the gradient at round $t$ must have been queried at round $t-1$, querying the naive stale point $y_{t-1}$ leads to the unstable delayed-Nesterov recurrence.  Instead, \cref{alg:delayed-comp-acsa} queries the one-step momentum-compensated point
\begin{equation}
\label{eq:comp-query}
    q_t
    :=
    x_{t-1}+m_{t-1}(1+m_t)(x_{t-1}-x_{t-2}),
    \qquad t\ge 1,
\end{equation}
with the convention $q_0=x_0$.

The extra term
\[
    m_{t-1}m_t(x_{t-1}-x_{t-2})
\]
is the one-step momentum compensation.  It anticipates the additional inertial displacement that will appear between $y_{t-1}$ and $y_t$.

\subsection{Main result}

\begin{theorem}[Optimal one-step delayed stochastic accelerated rate]
\label{thm:main}
Assume \cref{eq:noise-assumption,eq:square-integrable} hold.  Fix a horizon $T\ge 1$, and set
\[
    a_t = \frac{t+1}{128L},
    \qquad
    A_t = \sum_{i=0}^t a_i = \frac{(t+1)(t+2)}{256L},
    \qquad
    \theta_t = \frac{a_t}{A_t}=\frac{2}{t+2},
    \qquad 0\le t\le T.
\]
Set
\[
    m_0=0,
    \qquad
    m_t := \frac{\theta_t(1-\theta_{t-1})}{\theta_{t-1}}
    = \frac{t-1}{t+2},
    \qquad 1\le t\le T.
\]
Finally, define
\begin{equation}
\label{eq:beta-def}
    H_{T-1}:=\sqrt{\sum_{i=0}^{T-1} a_i^2},
    \qquad
    \beta
    :=
    1+\frac{\sigma}{R}H_{T-1}.
\end{equation}
Run \cref{alg:delayed-comp-acsa} with these parameters.  
Then
\begin{equation}
\label{eq:main-rate-explicit}
    \E[f(x_T)-f(x^\star)]
    \le
    \frac{6\sigma R}{\sqrt{T}}
    +
    \frac{128L R^2}{T^2}
    +
    \mathcal{E}_T,
\end{equation}
where 
\begin{equation*}
    \mathcal{E}_T :=
    2^{11} \left(R+\sqrt{\frac{\sigma R}{L}}\,T^{3/4}\right) T \Gamma_T
    +
    \frac{64}{L} T^4 \Gamma_T^2,
\end{equation*}
and $\Gamma_T$ is the weighted average RMS bias,
\begin{equation}
\label{eq:bias-budget}
    \Gamma_T:=
    \frac{1}{A_{T-1}}\sum_{t=0}^{T-1}a_t\epsilon_t.
\end{equation}
If $\epsilon_t=0$ for all $t$, then $\mathcal{E}_T=0$ and the last term in \cref{eq:main-rate-explicit} vanishes.
\end{theorem}

\paragraph{Relation to inexact-oracle acceleration.}
The bias dependence in \cref{thm:main} is consistent with the inexact-oracle analysis of \citet*{devolder2014first}.  Their accelerated bounds for a scalar $(\delta,L)$-inexact oracle contain an accumulated error term of order $T\delta$.  In the deterministic case $\sigma=0$, the leading bias term above is
$
    R T \Gamma_T,
$
which has the same form after interpreting the scalar oracle error as $\delta\sim R\Gamma_T$.  The additional linear term involving
$
    \sqrt{{\sigma R}/{L}}\,T^{1+3/4} \Gamma_T
$
comes from stochastic spread of the iterates.  The final quadratic term
$
    (1/L) T^4 \Gamma_T^2
$
is not a feature of the scalar inexact-oracle model itself; it appears here because we control vector-valued bias on an unconstrained domain.  The proof must first bound the RMS size of $z_t-x^\star$ in terms of the accumulated bias, and this self-consistency step produces the quadratic drift contribution.  Thus the \citet{devolder2014first}-type $O(T R\Gamma_T)$ accumulation is the leading effect, while the quadratic term is the price of converting vector bias into scalar function error without an a priori radius bound.

\subsection{Proofs}

\subsubsection{Elementary identities}

The following identity is the algebraic reason for the compensation. Without the compensated query point in (cf.~\cref{eq:comp-query}), the stale point would be $y_{t-1}$, and the difference $y_t-y_{t-1}$ contains an uncontrolled inertial component.  The scalar quadratic test then produces an unstable limiting recurrence.  The compensated point removes precisely the hard inertial component.  \cref{lem:comp-identity} below shows that the remaining mismatch is proportional only to the previous actual update $x_t-y_{t-1}$, which is controlled by the negative stability terms in the accelerated estimate.

\begin{lemma}[One-step compensation identity]
\label{lem:comp-identity}
For the parameter choice in \cref{thm:main} and every $t\ge 1$, \cref{alg:delayed-comp-acsa} satisfies
\begin{equation}
\label{eq:comp-identity}
    y_t-q_t=(1+m_t)(x_t-y_{t-1}).
\end{equation}
Consequently,
\begin{equation}
\label{eq:comp-norm}
    \norm{y_t-q_t}^2
    \le 4\norm{x_t-y_{t-1}}^2.
\end{equation}
\end{lemma}

\begin{proof}
Since
\[
    y_s=x_s+m_s(x_s-x_{s-1}),
\]
we have
\[
    x_t-y_{t-1}=x_t-x_{t-1}-m_{t-1}(x_{t-1}-x_{t-2}).
\]
Using the definition of $q_t$,
\begin{align*}
    y_t-q_t
    &=
    x_t+m_t(x_t-x_{t-1})
    -x_{t-1}-m_{t-1}(1+m_t)(x_{t-1}-x_{t-2})
    \\
    &=(1+m_t)(x_t-x_{t-1})
    -(1+m_t)m_{t-1}(x_{t-1}-x_{t-2})
    \\
    &=(1+m_t)(x_t-y_{t-1}).
\end{align*}
Since $0\le m_t<1$, $1+m_t\le 2$, giving \cref{eq:comp-norm}.
\end{proof}

\begin{lemma}[Weighted delay-energy bound]
\label{lem:weighted-delay}
For the parameter choice in \cref{thm:main} and every $T\ge 2$,
\begin{equation}
\label{eq:delay-energy}
    \sum_{t=1}^{T-1} A_t\norm{y_t-q_t}^2
    \le
    12\sum_{t=0}^{T-2} A_t\norm{x_{t+1}-y_t}^2.
\end{equation}
\end{lemma}

\begin{proof}
By \cref{lem:comp-identity},
\[
    \norm{y_t-q_t}^2\le 4\norm{x_t-y_{t-1}}^2.
\]
For $t\ge 1$,
\[
    \frac{A_t}{A_{t-1} }
    =
    \frac{(t+1)(t+2)}{t(t+1)}
    =
    \frac{t+2}{t}
    \le 3.
\]
Therefore
\[
    A_t\norm{y_t-q_t}^2
    \le
    12 A_{t-1}\norm{x_t-y_{t-1}}^2.
\]
Summing over $t=1,\ldots,T-1$ and re-indexing gives \cref{eq:delay-energy}.
\end{proof}

\subsubsection{Deterministic accelerated estimate with delayed gradients}

We first state the deterministic estimate-sequence inequality that will be used inside the stochastic proof.  Its role is to show that the compensated stale query behaves like an ordinary accelerated gradient step with a larger smoothness constant.

\begin{lemma}[Compensated delayed accelerated estimate]
\label{lem:det-estimate}
Let $T\ge 1$.  Let $g_t$ be any sequence of vectors, and define
\[
    \delta_t:=g_t-\nabla f(q_t).
\]
Suppose the iterates are generated by \cref{alg:delayed-comp-acsa} with the parameter choice in \cref{thm:main}.  Then, for every $u\in\R^d$,
\begin{align}
\label{eq:main-pathwise}
    A_{T-1}\bigl(f(x_T)-f(u)\bigr)
    +3L\sum_{t=0}^{T-1}A_t\norm{x_{t+1}-y_t}^2
    &\le
    \frac{\beta}{2}\norm{u-x_0}^2
    +
    \sum_{t=0}^{T-1} a_t\ip{\delta_t}{u-z_t}
    % \\
    % &\qquad
    +
    \sum_{t=0}^{T-1}\frac{a_t^2}{\beta}\norm{\delta_t}^2.
\end{align}
\end{lemma}

\begin{proof}
We give the proof in three steps.

\paragraph{Step 1: a smoothness inequality at the compensated point.}
Set $A_{-1}:=0$.  For each $t$, let
\[
    y_t=(1-\theta_t)x_t+\theta_t z_t,
    \qquad
    x_{t+1}=(1-\theta_t)x_t+\theta_t z_{t+1}.
\]
Smoothness at the query point $q_t$ gives
\[
    f(x_{t+1})
    \le
    f(q_t)+\ip{\nabla f(q_t)}{x_{t+1}-q_t}
    +\frac L2\norm{x_{t+1}-q_t}^2.
\]
By convexity, applied at $q_t$ and then averaged between $x_t$ and $u$,
\[
    (1-\theta_t)f(x_t)+\theta_t f(u)
    \ge
    f(q_t)+
    \ip{\nabla f(q_t)}{(1-\theta_t)x_t+\theta_t u-q_t}.
\]
Since $x_{t+1}=(1-\theta_t)x_t+\theta_t z_{t+1}$, the last two displays imply
\begin{align}
\label{eq:smooth-weighted}
    A_t f(x_{t+1})-A_{t-1}f(x_t)
    &\le
    a_t f(u)
    +a_t\ip{\nabla f(q_t)}{z_{t+1}-u}
    % \\
    % &\quad
    +\frac{L A_t}{2}\norm{x_{t+1}-q_t}^2.
\end{align}

\paragraph{Step 2: the fixed-prox FTRL inequality.}
The following argument establishes the FTRL inequality for a fixed quadratic regularizer centered at $x_0$; see \citet{shalev2011online,hazan2016introduction} for background on FTRL.
Let $s_t:=a_tg_t$ and $S_t:=\sum_{i=0}^t s_i$.  For $t\ge 0$, define
\[
    h_t(z):=\ip{S_t}{z}+\frac{\beta}{2}\norm{z-x_0}^2,
    \qquad z\in\R^d,
\]
and set $h_{-1}(z):=\frac{\beta}{2}\norm{z-x_0}^2$.  The recursion $z_{t+1}=z_t-s_t/\beta$ and $z_0=x_0$ imply $z_{t+1}=x_0-S_t/\beta$.  Thus $z_{t+1}$ minimizes $h_t$ on $\R^d$, and $z_0=x_0$ minimizes $h_{-1}$ on $\R^d$.

Because each $h_t$ is $\beta$-strongly convex on $\R^d$,
\[
    h_{t-1}(z_t)+\frac{\beta}{2}\norm{z_{t+1}-z_t}^2
    \le
    h_{t-1}(z_{t+1}).
\]
Using $h_t(z_{t+1})-h_{t-1}(z_{t+1})=\ip{s_t}{z_{t+1}}$, we obtain
\[
    \ip{s_t}{z_{t+1}}
    \le
    h_t(z_{t+1})-h_{t-1}(z_t)
    -\frac{\beta}{2}\norm{z_{t+1}-z_t}^2.
\]
Summing over $t=0,\ldots,T-1$ telescopes:
\[
    \sum_{t=0}^{T-1}\ip{s_t}{z_{t+1}}
    \le
    h_{T-1}(z_T)-h_{-1}(z_0)
    -\frac{\beta}{2}\sum_{t=0}^{T-1}\norm{z_{t+1}-z_t}^2.
\]
Since $z_T$ minimizes $h_{T-1}$ and $h_{-1}(z_0)=0$, for every $u\in\R^d$,
\begin{align}
\label{eq:ftrl}
    \sum_{t=0}^{T-1}a_t\ip{g_t}{z_{t+1}-u}
    \le
    \frac{\beta}{2}\norm{u-x_0}^2
    -
    \frac{\beta}{2}\sum_{t=0}^{T-1}\norm{z_{t+1}-z_t}^2.
\end{align}

\paragraph{Step 3: summation and absorption of delay energy.}
Substitute $\nabla f(q_t)=g_t-\delta_t$ into \cref{eq:smooth-weighted}, sum over $t=0,\ldots,T-1$, and use \cref{eq:ftrl}.  This gives
\begin{align*}
    A_{T-1}\bigl(f(x_T)-f(u)\bigr)
    &\le
    \frac{\beta}{2}\norm{u-x_0}^2
    +\sum_{t=0}^{T-1}a_t\ip{\delta_t}{u-z_{t+1}}
    \\
    &\quad
    -\frac{\beta}{2}\sum_{t=0}^{T-1}\norm{z_{t+1}-z_t}^2
    +\frac L2\sum_{t=0}^{T-1}A_t\norm{x_{t+1}-q_t}^2.
\end{align*}
Decompose
\[
    \ip{\delta_t}{u-z_{t+1}}
    =
    \ip{\delta_t}{u-z_t}+\ip{\delta_t}{z_t-z_{t+1}}.
\]
Young's inequality with parameter $\beta/2$ gives
\[
    a_t\ip{\delta_t}{z_t-z_{t+1}}
    \le
    \frac{a_t^2}{\beta}\norm{\delta_t}^2
    +
    \frac{\beta}{4}\norm{z_{t+1}-z_t}^2.
\]
Thus only the following residual deterministic term remains to be controlled:
\[
    \frac L2\sum_{t=0}^{T-1}A_t\norm{x_{t+1}-q_t}^2
    -
    \frac{\beta}{4}\sum_{t=0}^{T-1}\norm{z_{t+1}-z_t}^2.
\]

Since $x_{t+1}-y_t=\theta_t(z_{t+1}-z_t)$ and $\beta\ge 1$,
\[
    \frac{\beta}{4}\norm{z_{t+1}-z_t}^2
    =
    \frac{\beta}{4\theta_t^2}\norm{x_{t+1}-y_t}^2
    \ge
    16L A_t\norm{x_{t+1}-y_t}^2,
\]
where the last inequality follows from
\[
    \frac{1}{\theta_t^2}
    =
    \frac{(t+2)^2}{4}
    \ge
    64L A_t
    =
    \frac{(t+1)(t+2)}{4}.
\]
Moreover, $q_0=y_0=x_0$ and
\[
    \norm{x_{t+1}-q_t}^2
    \le
    2\norm{x_{t+1}-y_t}^2+2\norm{y_t-q_t}^2.
\]
Using \cref{lem:weighted-delay} when $T\ge 2$ and the trivial empty-sum bound when $T=1$,
\[
    \frac L2\sum_{t=0}^{T-1}A_t\norm{x_{t+1}-q_t}^2
    \le
    13L\sum_{t=0}^{T-1}A_t\norm{x_{t+1}-y_t}^2.
\]
Because $16L-13L=3L$, the residual deterministic term is bounded above by
\[
    -3L\sum_{t=0}^{T-1}A_t\norm{x_{t+1}-y_t}^2.
\]
Moving this term to the left proves \cref{eq:main-pathwise}.
\end{proof}

\subsubsection{A crude RMS-distance bound}

\begin{lemma}[Uniform RMS distance of query and dual points]
\label{lem:expected-distance}
Under the assumptions and parameter choice of \cref{thm:main}, define
\[
    B_T^\epsilon
    :=
    1024(T+1)\left(
        R+\sqrt{\frac{\sigma R}{L}}\,T^{3/4}
        +(T+1)A_{T-1}\Gamma_T
    \right).
\]
Then, for every $0\le t\le T$,
\[
    \left(\E\norm{q_t-x^\star}^2\right)^{1/2}\le B_T^\epsilon,
    \qquad
    \left(\E\norm{z_t-x^\star}^2\right)^{1/2}\le B_T^\epsilon.
\]
\end{lemma}

\begin{proof}
Let
\[
    d_t:=x_{t+1}-y_t.
\]
Set
\[
    S:=\sum_{t=0}^{T-1}A_t\E\norm{d_t}^2,
    \qquad
    D_t:=\sum_{i=0}^{t-1}\norm{d_i}.
\]
For $0\le s\le T$, set $v_s:=x_s-x_{s-1}$, with $v_0=0$.  For $0\le s\le T-1$, the identity
\[
    x_{s+1}=y_s+d_s=x_s+m_s(x_s-x_{s-1})+d_s
\]
implies
\[
    v_{s+1}=m_sv_s+d_s.
\]
Since $0\le m_s<1$, induction gives
\[
    \norm{v_s}\le \sum_{i=0}^{s-1}\norm{d_i} \le D_s,
    \qquad s\ge 1.
\]
Consequently, for $s\ge 0$,
\[
    \norm{x_s-x^\star}
    \le
    R+sD_s.
\]
For $t\ge 1$, using
\[
    q_t=x_{t-1}+m_{t-1}(1+m_t)v_{t-1}
\]
and $m_{t-1}(1+m_t)\le 2$ gives
\[
\begin{aligned}
    \norm{q_t-x^\star}
    &\le
    \norm{x_{t-1}-x^\star}+2\norm{v_{t-1}}
    % \\
    % &
    \le
    R+(t+1)D_t.
\end{aligned}
\]
Also, for $t\ge 1$,
\[
    x_t=(1-\theta_{t-1})x_{t-1}+\theta_{t-1}z_t,
\]
so
\[
    z_t=x_{t-1}+\frac{v_t}{\theta_{t-1}}.
\]
Since $\theta_{t-1}=2/(t+1)$,
\[
    \norm{z_t-x^\star}
    \le
    \norm{x_{t-1}-x^\star}+\frac{t+1}{2}\norm{v_t}
    \le
    R+\left(t-1+\frac{t+1}{2}\right)D_t
    \le
    R+2(t+1)D_t.
\]
The same two bounds hold at $t=0$ because $q_0=z_0=x_0$.

By Cauchy-Schwarz inequality and the explicit formula for $A_i$,
\[
    D_t^2
    \le
    \left(\sum_{i=0}^{t-1}\frac1{A_i}\right)
    \left(\sum_{i=0}^{t-1}A_i\norm{d_i}^2\right),
\]
and
\[
    \sum_{i=0}^{t-1}\frac1{A_i}
    =
    256L\sum_{i=0}^{t-1}\frac1{(i+1)(i+2)}
    \le
    256L.
\]
Therefore
\[
    \left(\E D_T^2\right)^{1/2}
    \le
    16\sqrt{LS}.
\]

We next bound $S$.  Apply \cref{lem:det-estimate} with $u=x^\star$ and $\delta_t=\zeta_t$, take expectations, and drop the nonnegative term $A_{T-1}\E[f(x_T)-f(x^\star)]$, giving
\[
    3LS
    \le
    \frac{\beta R^2}{2}
    +
    \sum_{t=0}^{T-1}a_t\E\ip{\zeta_t}{x^\star-z_t}
    +
    \frac1\beta\sum_{t=0}^{T-1}a_t^2\E\norm{\zeta_t}^2.
\]
The square-integrability assumption and the recursion for $z_t$ imply that each $z_t$ is square-integrable.
Since $z_t$ is also $\calF_t$-measurable and $b_t=\E[\zeta_t\mid\calF_t]$, this gives
\[
    3LS
    \le
    \frac{\beta R^2}{2}
    +
    \frac{\sigma^2}{\beta}H_{T-1}^2
    +
    \frac1\beta\sum_{t=0}^{T-1}a_t^2\epsilon_t^2
    +
    \sum_{t=0}^{T-1}a_t\epsilon_t
    \left(\E\norm{z_t-x^\star}^2\right)^{1/2}.
\]
Here we used the orthogonality
$\E\ip{\zeta_t-b_t}{b_t}=0$, which follows from
$b_t=\E[\zeta_t\mid\calF_t]$, to get
\[
    \E\norm{\zeta_t}^2
    =
    \E\norm{\zeta_t-b_t}^2+\E\norm{b_t}^2
    \le \sigma^2+\epsilon_t^2
\]
and Cauchy-Schwarz inequality in expectation:
\[
    \E\ip{b_t}{x^\star-z_t}
    \le
    \left(\E\norm{b_t}^2\right)^{1/2}
    \left(\E\norm{z_t-x^\star}^2\right)^{1/2}.
\]
Using the pathwise bound on $z_t$ and the preceding estimate on $\left(\E D_T^2\right)^{1/2}$,
\[
    \sum_{t=0}^{T-1}a_t\epsilon_t
    \left(\E\norm{z_t-x^\star}^2\right)^{1/2}
    \le
    A_{T-1}\Gamma_T\left(R+32(T+1)\sqrt{LS}\right).
\]
Also,
\[
    \frac1\beta\sum_{t=0}^{T-1}a_t^2\epsilon_t^2
    \le
    \left(\sum_{t=0}^{T-1}a_t\epsilon_t\right)^2
    =
    A_{T-1}^2\Gamma_T^2.
\]
Moreover,
\[
    \frac{\beta R^2}{2}
    +
    \frac{\sigma^2}{\beta}H_{T-1}^2
    \le
    \frac{R^2+3\sigma R H_{T-1}}{2},
\]
because $\beta=1+\sigma H_{T-1}/R$ and $\sigma H_{T-1}/\beta\le R$.  With $Y:=\sqrt{LS}$, we have
\[
    3Y^2
    \le
    \frac{R^2+3\sigma R H_{T-1}}{2}
    +A_{T-1}^2\Gamma_T^2
    +RA_{T-1}\Gamma_T
    +32(T+1)A_{T-1}\Gamma_TY.
\]
The quadratic inequality yields
\[
    Y
    \le
    \frac{32}{3}(T+1)A_{T-1}\Gamma_T
    +
    \sqrt{\frac{R^2+3\sigma R H_{T-1}+2RA_{T-1}\Gamma_T+2A_{T-1}^2\Gamma_T^2}{6}}.
\]
We simplify this bound using
\[
    \sqrt{x_1+\cdots+x_4}\le \sqrt{x_1}+\cdots+\sqrt{x_4},
    \qquad
    \frac{32}{3}\le 16,
    \qquad
    H_{T-1}\le \frac{T^{3/2}}{64\sqrt3L},
\]
and
\[
    \sqrt{RA_{T-1}\Gamma_T}
    \le R+(T+1)A_{T-1}\Gamma_T.
\]
We also use $A_{T-1}\Gamma_T\le (T+1)A_{T-1}\Gamma_T$.
This gives
\[
    Y
    \le
    20\left(
        R+\sqrt{\frac{\sigma R}{L}}\,T^{3/4}
        +(T+1)A_{T-1}\Gamma_T
    \right).
\]
Combining this with the bounds
\[
    \left(\E\norm{q_t-x^\star}^2\right)^{1/2}\le R+16(T+1)Y,
    \qquad
    \left(\E\norm{z_t-x^\star}^2\right)^{1/2}\le R+32(T+1)Y
\]
and using $T\ge 1$ proves the lemma by the definition of $B_T^\epsilon$.
\end{proof}

\subsubsection{Proof of \texorpdfstring{\cref{thm:main}}{Theorem 3}}

\begin{proof}[\unskip\nopunct]%[Proof of \cref{thm:main}]
Apply \cref{lem:det-estimate} with $u=x^\star$ and $\delta_t=\zeta_t$, and drop the nonnegative movement term.  We obtain
\begin{align*}
    A_{T-1}\bigl(f(x_T)-f(x^\star)\bigr)
    &\le
    \frac{\beta}{2}\norm{x^\star-x_0}^2
    +
    \sum_{t=0}^{T-1}a_t\ip{\zeta_t}{x^\star-z_t}
    % \\
    % &\qquad
    +
    \sum_{t=0}^{T-1}\frac{a_t^2}{\beta}\norm{\zeta_t}^2.
\end{align*}
The square-integrability assumption implies by induction that each $z_t=x_0-\beta^{-1}\sum_{i=0}^{t-1}a_iG_i$ is square-integrable.  Hence the inner products below are integrable.  Because $z_t$ is $\calF_t$-measurable and $b_t=\E[\zeta_t\mid\calF_t]$,
\[
    \E\left[a_t\ip{\zeta_t}{x^\star-z_t}\right]
    =
    \E\left[a_t\ip{b_t}{x^\star-z_t}\right]
    \le
    a_t\epsilon_t\left(\E\norm{z_t-x^\star}^2\right)^{1/2}.
\]
By \cref{lem:expected-distance},
\[
    \sum_{t=0}^{T-1}
    a_t\epsilon_t\left(\E\norm{z_t-x^\star}^2\right)^{1/2}
    \le
    A_{T-1}\Gamma_TB_T^\epsilon.
\]
Also, since $b_t=\E[\zeta_t\mid\calF_t]$, the same orthogonality gives
\[
    \E\norm{\zeta_t}^2
    =
    \E\norm{\zeta_t-b_t}^2+\E\norm{b_t}^2
    \le
    \sigma^2+\epsilon_t^2.
\]
Therefore
\begin{equation}
\label{eq:pre-rate}
    \E[f(x_T)-f(x^\star)]
    \le
    \frac{\beta R^2}{2A_{T-1}}
    +
    \frac{\sigma^2}{A_{T-1}\beta}
      \sum_{t=0}^{T-1}a_t^2
    +
    \frac{1}{A_{T-1}\beta}
      \sum_{t=0}^{T-1}a_t^2\epsilon_t^2
    +
    \Gamma_TB_T^\epsilon.
\end{equation}
By definition,
\[
    \beta=1+\frac{\sigma}{R}H_{T-1}.
\]
Thus
\[
    \frac{\sigma^2}{\beta}\sum_{t=0}^{T-1}a_t^2
    =
    \frac{\sigma^2}{\beta}H_{T-1}^2
    \le
    \sigma R H_{T-1},
\]
because $\sigma H_{T-1}/\beta\le R$.
Also,
\[
    \frac{1}{\beta}\sum_{t=0}^{T-1}a_t^2\epsilon_t^2
    \le
    \left(\sum_{t=0}^{T-1}a_t\epsilon_t\right)^2
    =
    A_{T-1}^2\Gamma_T^2.
\]
Substituting into \cref{eq:pre-rate} gives
\[
    \E[f(x_T)-f(x^\star)]
    \le
    \frac{R^2}{2A_{T-1}}
    +
    \frac{\sigma R H_{T-1}}{2A_{T-1}}
    +
    \frac{\sigma R H_{T-1}}{A_{T-1}}
    +
    A_{T-1}\Gamma_T^2
    +
    \Gamma_TB_T^\epsilon.
\]
Hence
\begin{equation}
\label{eq:rate-H}
    \E[f(x_T)-f(x^\star)]
    \le
    \frac{R^2}{2A_{T-1}}
    +
    \frac{3\sigma R H_{T-1}}{2A_{T-1}}
    +
    A_{T-1}\Gamma_T^2
    +
    \Gamma_TB_T^\epsilon.
\end{equation}
Now compute
\[
    A_{T-1}=\frac{T(T+1)}{256L}
    \ge
    \frac{T^2}{256L},
\]
and
\[
    H_{T-1}^2
    =
    \sum_{t=0}^{T-1}\frac{(t+1)^2}{16384L^2}
    =
    \frac{T(T+1)(2T+1)}{98304L^2}
    \le
    \frac{T^3}{12288L^2}.
\]
Therefore
\[
    \frac{R^2}{2A_{T-1}}
    =
    \frac{128L R^2}{T(T+1)},
\]
and
\[
    \frac{3\sigma R H_{T-1}}{2A_{T-1}}
    \le
    \frac{3\sigma R}{2}\cdot
    \frac{T^{3/2}/(64\sqrt 3L)}{T(T+1)/(256L)}
    \le
    \frac{6\sigma R}{\sqrt{3T}}
    .
\]
Finally, by the definitions of $B_T^\epsilon$ and $A_{T-1}$,
\begin{align*}
    A_{T-1}\Gamma_T^2+\Gamma_TB_T^\epsilon
    &=
    2^{10}(T+1)\Gamma_T
    \left(
        R+\sqrt{\frac{\sigma R}{L}}\,T^{3/4}
        +(T+1)A_{T-1}\Gamma_T
    \right)
    +
    A_{T-1}\Gamma_T^2
    \\
    &=
    2^{10}(T+1)\Gamma_T
    \left(R+\sqrt{\frac{\sigma R}{L}}\,T^{3/4}\right)
    +
    \left(\frac{4T(T+1)^3}{L}+\frac{T(T+1)}{256L}\right)\Gamma_T^2
    \\
    &\le
    2^{11}T\Gamma_T
    \left(R+\sqrt{\frac{\sigma R}{L}}\,T^{3/4}\right)
    +
    \frac{64}{L}T^4\Gamma_T^2.
\end{align*}
Combining these estimates proves \cref{eq:main-rate-explicit}.
\end{proof}

\section{Accelerated gossip via Nesterov acceleration}
\label{app:nesterov-gossip}

This appendix records the accelerated communication primitive used in the distributed
algorithm.  
% A gossip round is an averaging step over the network, and repeated
% ordinary gossip contracts disagreement at a rate governed by the spectral gap
% \(\rho\).  
The standard way to accelerate gossip is to apply a Chebyshev
polynomial in the gossip matrix, which improves the communication round complexity
from \(\rho^{-1}\) to \(\rho^{-1/2}\). Here we give, for completeness, an alternative derivation 
from an optimization viewpoint: the accelerated version below
applies Nesterov's method to a quadratic disagreement objective, giving the
same \(\rho^{-1/2}\) dependence while preserving the network average exactly.
We present the argument for scalar local variables; matrix gossip in the
main algorithm follows by applying the same recursion row by row.

\paragraph{Setup.}
Let \(P\in\mathbb{R}^{n\times n}\) be a symmetric doubly-stochastic gossip
matrix on a connected graph.  In this section each node holds a scalar, so the
global state is a row vector \(x\in\mathbb R^{1\times n}\), identified with
\(\mathbb R^n\).  Gossip acts by right multiplication, matching the convention
used in the main algorithm.  When the main algorithm gossips a matrix of
vector-valued local copies, the vector result below is applied independently to
each row, giving the corresponding Frobenius-norm contraction.  Let
\[
    J:=\frac1n\mathbf{1}\mathbf{1}^\top .
\]
Thus, for an initial row vector \(\xi\in\mathbb R^{1\times n}\), the vector
\(\xi J\) has all coordinates equal to the average of the coordinates of \(\xi\).
Write the eigenvalues of \(P\) as
\(1=\lambda_1(P)>\lambda_2(P)\geq \cdots \geq \lambda_n(P)\), and assume that the
spectral gap satisfies \(1-\lambda_2(P) \geq \rho\).  Consider the
quadratic disagreement objective
\[
    F(x):=\frac12 x(I-P)x^\top.
\]
The minimizers of \(F\) are exactly the consensus vectors \(x=xJ\).  On the
row-disagreement subspace \(\{e:eJ=0\}\), the matrix \(I-P\) has eigenvalues
\(1-\lambda_i(P)\), \(i=2,\dots,n\), and hence \(F\) is \(\rho\)-strongly
convex and \(2\)-smooth with respect to the Euclidean norm.

\paragraph{Algorithm.}

Applying Nesterov's accelerated gradient method to the quadratic disagreement
objective \(F(x)=\frac12x(I-P)x^\top\) on the row-disagreement subspace gives
the following algorithm.

\begin{algorithm}[H]
\caption{Nesterov-accelerated gossip}
\label{alg:nag-gossip}
\begin{algorithmic}[1]
\STATE {\bf input:} Symmetric doubly-stochastic \(P\), initial vector \(\xi\), stepsize \(\eta>0\), and momentum \(\beta\ge 0\).
\STATE Initialize \(x_{-1}=x_0=\xi\).
\FOR{\(t=0,1,2,\dots\)}
    \STATE \(y_t = x_t+\beta(x_t-x_{t-1})\).
    \STATE \(x_{t+1}=y_t-\eta y_t(I-P)=(1-\eta)y_t+\eta y_tP\).
\ENDFOR
\end{algorithmic}
\end{algorithm}

Equivalently, each node performs
\[
    y_{i,t}=x_{i,t}+\beta(x_{i,t}-x_{i,t-1}),
    \qquad
    x_{i,t+1}
    =
    (1-\eta)y_{i,t}
    +
    \eta\sum_j P_{ji}y_{j,t}.
\]
Since \(P\) is symmetric, this is the usual weighted neighbor average.

% \paragraph{Relation to Chebyshev gossip.}

Unlike classical Chebyshev gossip, which explicitly constructs the minimax
degree-\(t\) polynomial on the disagreement spectrum, the recursion above arises
from Nesterov acceleration on the quadratic \(F(x)=\frac12x(I-P)x^\top\).  The
resulting polynomial need not have Chebyshev-optimal constants, but it gives the
same \(O(\rho^{-1/2}\log(1/\varepsilon))\) communication dependence.

\paragraph{Convergence analysis.}

By directly utilizing standard convergence bounds for Nesterov's method in the strongly convex and smooth case, one can obtain the following.

\begin{theorem}[Accelerated gossip convergence]
\label{thm:nesterov-gossip-convergence}
Let \(P\) be symmetric and doubly stochastic, and suppose that the
non-consensus eigenvalues of \(I-P\) lie in \([\rho,2]\) for some
\(\rho\in(0,1]\). Run \cref{alg:nag-gossip} with
\[
    \eta=\frac12,
    \qquad
    \beta=\frac{1-\sqrt{\rho/2}}{1+\sqrt{\rho/2}}.
\]
Then, for every
\(t\geq 0\) it holds that 
$
    x_tJ=\xi J,
$
and
\[
    \|x_t-\xi J\|_2
    \leq
    \frac{2}{\sqrt\rho}
    \left(1-\sqrt{\rho/2}\right)^{t/2}
    \|\xi-\xi J\|_2
    .
\]
Consequently, \cref{alg:nag-gossip} reaches
\(\|x_t-\xi J\|_2\leq \varepsilon\) with iteration complexity
\[
    O\!\left(
        \rho^{-1/2}
        \log\!\left(
            \frac{\|\xi-\xi J\|_2}{\varepsilon\sqrt\rho}
        \right)
    \right).
\]
\end{theorem}

We note that \citet[Proposition~1]{ye2023multi} provide a similar accelerated-consensus guarantee for a different recursion, which could serve as an alternative gossip primitive in our framework after adapting the gossip updates. Applied to the lazy matrix $W=(I+P)/2$, their bound offers a slightly different tradeoff between the prefactor and exponential contraction.

\begin{proof}
First, the average is invariant. Since \(P\) is doubly stochastic,
\((I-P)J=0\). Therefore \(x_{t+1}J=y_tJ\). Also,
\[
    y_tJ=x_tJ+\beta(x_tJ-x_{t-1}J).
\]
Since \(x_{-1}=x_0=\xi\), induction gives \(x_tJ=\xi J\) for all \(t\).

Define the disagreement vector \(e_t:=x_t-\xi J\). The previous paragraph shows
that \(e_tJ=0\) for all \(t\). Moreover, since \(J(I-P)=0\), we have
\(\xi J(I-P)=0\). Subtracting \(\xi J\) from the update gives
\[
    e_{t+1}
    =
    \widetilde y_t-\frac12\widetilde y_t (I-P),
    \qquad
    \widetilde y_t
    =
    e_t+\beta(e_t-e_{t-1}).
\]
Thus the disagreement dynamics are exactly Nesterov's accelerated gradient
method applied to the quadratic \(F(e):=\frac12 e(I-P)e^\top\) on the
row-disagreement subspace.

On this subspace, the eigenvalues of \(I-P\) lie in \([\rho,2]\).
Hence \(F\) is \(\rho\)-strongly convex and \(2\)-smooth with respect
to the Euclidean norm. By the standard finite-time
guarantee for Nesterov's accelerated gradient method on a
\(\mu\)-strongly convex and \(L\)-smooth objective, with
\(\eta=1/L\) and
\(\beta=(1-\sqrt{\mu/L})/(1+\sqrt{\mu/L})\),
\[
    F(e_t)-F(e^\star)
    \leq
    \frac{L+\mu}{2} \left(1-\sqrt{\frac{\mu}{L}}\right)^t \|e_0-e^\star\|_2^2
    .
\]
Applying this bound with \(\mu=\rho\), \(L=2\),
\(e^\star=0\), and \(e_0=\xi-\xi J\), we obtain
\[
    F(e_t)
    \leq
    2 \left(1-\sqrt{\rho/2}\right)^t \|\xi-\xi J\|_2^2.
\]
On the other hand, strong convexity on the row-disagreement subspace gives
\(F(e_t)\geq \frac{\rho}{2}\|e_t\|_2^2\). Combining the two inequalities
yields
\[
    \|x_t-\xi J\|_2
    \leq
    \frac{2}{\sqrt\rho}
    \left(1-\sqrt{\rho/2}\right)^{t/2}
    \|\xi-\xi J\|_2
    .
\]
The iteration-complexity statement follows by taking logarithms and using
\((1-a)^k\le \exp(-ak)\) for \(a\in(0,1)\).
\end{proof}

\section{Proofs for Decentralized Doubly Accelerated SGD (\texorpdfstring{\dadq}{DDA-SGD})}
\label{app:distributed-delayed-gossip}

For the analysis, define the network averages
\[
    \bar x_t:=\frac1M X_t\mathbf 1,
    \qquad
    \bar z_t:=\frac1M Z_t\mathbf 1,
    \qquad
    \bar q_t:=\frac1M Q_t\mathbf 1,
\]
and the averaged stochastic oracle
\[
    \widehat G_t:=\frac1M\sum_{i=1}^M g_{t,i}.
\]
We write
\[
    \zeta_t:=\widehat G_t-\nabla f(\bar q_t),
    \qquad
    b_t:=\E[\zeta_t\mid\calF_t].
\]

Before proving the main convergence result, we give some preliminary lemmas.
First we give a bound on the heterogeneity of the local gradients away from the optimum, which is used to control the bias of the averaged oracle induced by delayed gossip.

\begin{lemma}[Heterogeneity away from the optimum]
\label{lem:heterogeneity-away}
Under \cref{eq:bounded-heterogeneity}, every square-integrable random vector $X$ satisfies
\[
    \left(
        \E\left[
            \frac1M\sum_{i=1}^M
            \norm{\nabla f_i(X)-\nabla f(X)}^2
        \right]
    \right)^{1/2}
    \le
    \zeta_\star+2L\left(\E\norm{X-x^\star}^2\right)^{1/2}.
\]
\end{lemma}

\begin{proof}
For any deterministic $x\in\R^d$,
\begin{align*}
    &\left(
        \frac1M\sum_{i=1}^M
        \norm{\nabla f_i(x)-\nabla f(x)}^2
    \right)^{1/2}
    \\
    &\quad\le
    \left(
        \frac1M\sum_{i=1}^M
        \norm{\nabla f_i(x^\star)-\nabla f(x^\star)}^2
    \right)^{1/2}
    \\
    &\qquad
    +
    \left(
        \frac1M\sum_{i=1}^M
        \norm{\nabla f_i(x)-\nabla f_i(x^\star)}^2
    \right)^{1/2}
    +
    \norm{\nabla f(x)-\nabla f(x^\star)}.
\end{align*}
The first term is at most $\zeta_\star$ by \cref{eq:bounded-heterogeneity}; the last two terms are at most $L\norm{x-x^\star}$ by $L$-smoothness of the $f_i$'s and of $f$.  Therefore the deterministic heterogeneity at $x$ is at most $\zeta_\star+2L\norm{x-x^\star}$.  Applying this pointwise with $x=X$ and using the triangle inequality in $L_2$ gives the claim.
\end{proof}

Next we show that the network averages satisfy the same delayed accelerated recursions as the single-node algorithm (\cref{alg:delayed-comp-acsa}), with a single stochastic oracle given by the average of the local stochastic gradients.  We also give bounds on the bias and variance of this averaged oracle in terms of the network disagreement.

\begin{lemma}[Average oracle induced by delayed gossip]
\label{lem:distributed-oracle}
The network averages satisfy
\[
    \bar q_{t+1}
    =
    \bar x_t+m_t(1+m_{t+1})(\bar x_t-\bar x_{t-1}),
    \qquad
    \bar z_{t+1}
    =
    \bar z_t-\eta_t\widehat G_t,
\]
\[
    \bar x_{t+1}
    =
    (1-\theta_t)\bar x_t+\theta_t\bar z_{t+1}.
\]
In particular, when $\eta_t=a_t/\beta$ for all $t$, these are the same delayed accelerated recursions as \cref{alg:delayed-comp-acsa}, with stochastic oracle $\widehat G_t$ and query point $\bar q_t$. Moreover,
\begin{equation}
\label{eq:distributed-bias-variance}
    \norm{b_t}
    \le
    \frac{L}{\sqrt M}\norm{Q_t-\bar Q_t}_F,
    \qquad
    \E[\norm{\zeta_t-b_t}^2\mid\calF_t]
    \le
    \frac{\sigma^2}{MB}.
\end{equation}
\end{lemma}

\begin{proof}
Since $PJ=J$, each accelerated gossip block preserves network averages.  Thus
\[
    \bar q_{t+1}
    =
    \bar x_t+m_t(1+m_{t+1})(\bar x_t-\bar x_{t-1}).
\]
Also, by averaging the local updates,
\[
    \bar z_{t+1}
    =
    \bar z_t-\eta_t\widehat G_t,
    \qquad
    \bar x_{t+1}
    =
    (1-\theta_t)\bar x_t+\theta_t\bar z_{t+1}.
\]
When $\eta_t=a_t/\beta$ for all $t$, these are precisely the single-node delayed accelerated recursions for the averaged variables.

The conditional bias is
\[
    b_t
    =
    \frac1M\sum_{i=1}^M
    \left(\nabla f_i(q_{t,i})-\nabla f_i(\bar q_t)\right).
\]
By $L$-smoothness and Cauchy-Schwarz inequality,
\[
    \norm{b_t}
    \le
    \frac{L}{M}\sum_{i=1}^M\norm{q_{t,i}-\bar q_t}
    \le
    \frac{L}{\sqrt M}\norm{Q_t-\bar Q_t}_F.
\]
The centered stochastic part is the average of $MB$ conditionally independent zero-mean samples, each with conditional second moment at most $\sigma^2$, hence its conditional second moment is at most $\sigma^2/(MB)$.
\end{proof}

Finally, we give a bound on the network bias and output disagreement induced by delayed gossip.  The bound depends on the initial disagreement, the heterogeneity of the local gradients away from the optimum, and the pathwise bound on the network average.

\begin{lemma}[RMS network bias and output disagreement from delayed gossip]
\label{lem:distributed-disagreement}
Assume \cref{eq:bounded-heterogeneity}.  Suppose the algorithmic parameters
obey, for \(0\le t<T\),
\[
    \eta_t\ge 0,
    \qquad
    0\le \theta_t\le 1,
    \qquad
    0\le m_t(1+m_{t+1})\le 2,
\]
and that each accelerated gossip block uses
\[
    \alpha_{\rm g}=\frac12,
    \qquad
    \beta_{\rm g}
    =
    \frac{1-\sqrt{\rho/2}}{1+\sqrt{\rho/2}}.
\]
Let
\[
    Q_T^{\rm rms}:=
    \max_{0\le t<T}
    \left(\E\norm{\bar q_t-x^\star}^2\right)^{1/2},
    \qquad
    S_T^\eta:=\sum_{i=0}^{T-1}\eta_i,
    \qquad
    \tau:=\frac{2}{\sqrt\rho}
    \left(1-\sqrt{\rho/2}\right)^{B/2}.
\]
Define the initial disagreement quantities
\[
    D_{\rm init}^{\rm rms}
    :=
    \left(\E\norm{X_0-\bar X_0}_F^2\right)^{1/2}
    +
    \left(\E\norm{Z_0-\bar Z_0}_F^2\right)^{1/2},
\]
\[
    Q_{\rm init}^{\rm rms}
    :=
    \max\left\{
    \left(\E\norm{Q_0-\bar Q_0}_F^2\right)^{1/2},
    \left(\E\norm{\widetilde X_{-1}-\bar{\widetilde X}_{-1}}_F^2\right)^{1/2}
    \right\},
    \qquad
    \Delta_{\rm init}^{\rm rms}:=
    D_{\rm init}^{\rm rms}+3LS_T^\eta Q_{\rm init}^{\rm rms}.
\]
If
\begin{equation}
\label{eq:network-bias-absorb-condition}
    5LS_T^\eta\tau\le \frac12,
\end{equation}
then, for every $0\le t<T$,
\begin{equation}
\label{eq:network-bias-direct-bound}
    \left(\E\norm{b_t}^2\right)^{1/2}
    \le
    10LS_T^\eta\tau
    \left(\zeta_\star+\frac{\sigma}{\sqrt B}+2LQ_T^{\rm rms}\right)
    +
    \frac{10L\tau\Delta_{\rm init}^{\rm rms}+3LQ_{\rm init}^{\rm rms}}{\sqrt M}.
\end{equation}
Moreover, the final gossiped output satisfies
\begin{equation}
\label{eq:gossiped-output-disagreement}
    \left(\E\norm{\widetilde X_T-\bar X_T}_F^2\right)^{1/2}
    \le
    2S_T^\eta\tau\sqrt M
    \left(\zeta_\star+\frac{\sigma}{\sqrt B}+2LQ_T^{\rm rms}\right)
    +
    2\tau\Delta_{\rm init}^{\rm rms}.
\end{equation}
\end{lemma}

\begin{proof}
Define
\[
    D_T^{\rm rms}
    :=
    \max_{0\le t\le T}
    \left(\E\norm{X_t-\bar X_t}_F^2\right)^{1/2}.
\]
By \cref{eq:gossip-spectral-gap}, the non-consensus eigenvalues of $I-P$ lie in $[\rho,2]$.  Thus \cref{thm:nesterov-gossip-convergence} applies directly.  Applying it row by row to each accelerated gossip block gives preservation of averages and the Frobenius contraction factor $\tau$.
Let $J:=\mathbf 1\mathbf 1^\top/M$ and write $W^\perp:=W(I-J)=W-\bar W$ for the disagreement component of any matrix $W\in\R^{d\times M}$.  First we bound the local gradient disagreement.  Since $G_t^\perp$ is the distance from $G_t$ to the consensus subspace,
\[
    \norm{G_t^\perp}_F
    \le
    \norm{G_t-\nabla f(\bar q_t)\mathbf 1^\top}_F .
\]
Writing $g_{t,i}=B^{-1}\sum_{s=1}^B G_i(q_{t,i},\xi_{t,i}^{(s)})$, the $i$th column of the matrix on the right is
\[
    g_{t,i}-\nabla f(\bar q_t)
    =
    \bigl(g_{t,i}-\nabla f_i(q_{t,i})\bigr)
    +
    \bigl(\nabla f_i(q_{t,i})-\nabla f_i(\bar q_t)\bigr)
    +
    \bigl(\nabla f_i(\bar q_t)-\nabla f(\bar q_t)\bigr).
\]
For $t\ge 1$, since
\[
    Q_t
    =
    \widetilde X_{t-1}
    +
    m_{t-1}(1+m_t)(\widetilde X_{t-1}-\widetilde X_{t-2})
\]
and $0\le m_{t-1}(1+m_t)\le 2$, the accelerated-gossip contraction gives
\[
    \left(\E\norm{Q_t-\bar Q_t}_F^2\right)^{1/2}
    \le
    5\tau D_T^{\rm rms}
    +
    2Q_{\rm init}^{\rm rms}.
\]
The same bound, with the right-hand side enlarged to
\(5\tau D_T^{\rm rms}+3Q_{\rm init}^{\rm rms}\), also holds at \(t=0\).
By the triangle inequality in $L_2$, the minibatch variance bound, $L$-smoothness, the preceding display, and \cref{lem:heterogeneity-away},
\begin{equation}
\label{eq:local-gradient-disagreement}
    \left(\E\norm{G_t^\perp}_F^2\right)^{1/2}
    \le
    \sqrt M\,\frac{\sigma}{\sqrt B}
    +
    5L\tau D_T^{\rm rms}
    +
    3LQ_{\rm init}^{\rm rms}
    +
    \sqrt M\left(\zeta_\star+2LQ_T^{\rm rms}\right).
\end{equation}

We now express primal disagreement through past gradient disagreements:
\[
    Z_{t+1}^\perp
    =
    Z_t^\perp-\eta_tG_t^\perp,
    \qquad
    X_{t+1}^\perp
    =
    (1-\theta_t)X_t^\perp+\theta_tZ_{t+1}^\perp.
\]
Because $0\le\theta_t\le 1$,
\[
    \norm{X_t^\perp}_F
    \le
    \norm{X_0^\perp}_F+\norm{Z_0^\perp}_F
    +
    \sum_{i=0}^{t-1}\eta_i\norm{G_i^\perp}_F,
    \qquad 0\le t\le T,
\]
pathwise.
Taking $L_2$ norms, using the triangle inequality (in $L_2$), and applying \cref{eq:local-gradient-disagreement} gives
\[
    D_T^{\rm rms}
    \le
    D_{\rm init}^{\rm rms}
    +
    S_T^\eta
    \left(
        5L\tau D_T^{\rm rms}
        +
        3LQ_{\rm init}^{\rm rms}
        +
        \sqrt M\left(\zeta_\star+\frac{\sigma}{\sqrt B}+2LQ_T^{\rm rms}\right)
    \right).
\]
Under \cref{eq:network-bias-absorb-condition}, moving the self-coupling term to the left gives
\[
    D_T^{\rm rms}
    \le
    2\Delta_{\rm init}^{\rm rms}
    +
    2S_T^\eta
    \sqrt M\left(\zeta_\star+\frac{\sigma}{\sqrt B}+2LQ_T^{\rm rms}\right).
\]

Since $\norm{\widetilde X_T-\bar X_T}_F\le \tau\norm{X_T-\bar X_T}_F$, this proves \cref{eq:gossiped-output-disagreement}.  Finally, for every $0\le t<T$, by \cref{eq:distributed-bias-variance} and the query-disagreement bound above,
\[
    \norm{b_t}
    \le
    \frac{L}{\sqrt M}\norm{Q_t-\bar Q_t}_F
    \le
    \frac{5L\tau}{\sqrt M}D_T^{\rm rms}
    +
    \frac{3LQ_{\rm init}^{\rm rms}}{\sqrt M},
\]
and using the preceding bound on $D_T^{\rm rms}$ proves \cref{eq:network-bias-direct-bound}.
\end{proof}

We proceed to prove the main convergence result for \cref{alg:distributed-delayed-gossip}.

\subsection{Proof of \texorpdfstring{\cref{thm:distributed-gossip}}{Theorem 1}}

\begin{proof}
\medskip
\noindent\textbf{Stage 1: Reduction to a single-node problem.}
By \cref{lem:distributed-oracle}, the averaged dynamics satisfy the
single-node delayed accelerated recursion analyzed in \cref{thm:main}.  Namely,
with averaged oracle
\(\widehat G_t=\nabla f(\bar q_t)+\zeta_t\), the sequence
\((\bar q_t,\bar z_t,\bar x_t)\) satisfies, for \(0\le t<T\),
\[
    \bar q_{t+1}
    =
    \bar x_t+m_t(1+m_{t+1})(\bar x_t-\bar x_{t-1}),
    \qquad
    \bar z_{t+1}
    =
    \bar z_t-\eta_t\widehat G_t,
\]
\[
    \bar x_{t+1}
    =
    (1-\theta_t)\bar x_t+\theta_t\bar z_{t+1}.
\]
The same lemma gives the centered-variance part of
\cref{eq:noise-assumption}, namely
\[
    \E[\norm{\zeta_t-b_t}^2\mid\calF_t]\le \frac{\sigma^2}{MB}.
\]
Since the chosen parameters have \(\eta_t=a_t/\beta\), and since the
consensual initialization gives
\(\bar x_{-1}=\bar x_0=\bar z_0=\bar q_0=x_0\), the assumption
\(R\ge\norm{x_0-x^\star}\) gives the radius condition
\(R\ge\norm{\bar x_0-x^\star}\).  Thus the averaged dynamics have the
initialization and recursion required by \cref{thm:main}.

\medskip
\noindent\textbf{Stage 2: Verifying the hypotheses of \cref{thm:main}.}
In order to invoke \cref{thm:main}, it remains to verify the following
two properties:
\begin{enumerate}[label=(\roman*),leftmargin=*]
    \item The averaged oracle calls are square-integrable, as required in
    \cref{eq:square-integrable};
    \item The averaged oracle admits deterministic RMS bias bounds
    \((\epsilon_t)_{t=0}^{T-1}\), completing \cref{eq:noise-assumption};
\end{enumerate}

We first prove property~(i).  The optimality condition gives
\(\nabla f(x^\star)=0\), so \cref{eq:bounded-heterogeneity} bounds the average
squared norm of the vectors \(\nabla f_i(x^\star)\).  Together with
\(L\)-smoothness and the local variance assumption, this implies by induction
that all local iterates, query points, and averaged oracle calls generated up
to time \(T\) are square-integrable.

We next prove property~(ii).  The goal is to show that the RMS oracle
bias \((\E\norm{b_t}^2)^{1/2}\) is uniformly bounded for all \(0\le t<T\). Let
\[
    \tau:=\frac{2}{\sqrt\rho}
    \left(1-\sqrt{\rho/2}\right)^{B/2},
    \qquad
    S_T^\eta:=\sum_{t=0}^{T-1}\eta_t,
    \qquad
    Q_T^{\rm rms}
    :=
    \max_{0\le t<T}
    \left(\E\norm{\bar q_t-x^\star}^2\right)^{1/2}.
\]
The chosen parameters satisfy $\eta_t\ge0$, $0\le \theta_t\le1$, and
$0\le m_t(1+m_{t+1})\le2$ for $0\le t<T$, and the gossip parameters are those
required by \cref{lem:distributed-disagreement}.  The initialization in
\cref{alg:distributed-delayed-gossip} is consensual, so
\(D_{\rm init}^{\rm rms}=Q_{\rm init}^{\rm rms}
=\Delta_{\rm init}^{\rm rms}=0\) in
\cref{lem:distributed-disagreement}.  Also, \(Q_0\) is consensual, so
\(b_0=0\) by \cref{eq:distributed-bias-variance}.  Thus
\cref{lem:distributed-disagreement} will apply once we check that
\(5LS_T^\eta\tau\le 1/2\).

To check this inequality, put
\[
    U_N:=\frac{N+M}{M},
    \qquad
    \lambda_N:=\frac{32 N}{\sqrt M}
    \left(
        1+
        \frac{\zeta_\star+\sigma}{LR}
    \right).
\]
By definition, $\Lambda_N\ge e\lambda_N$.  Since $B\ge 1$ and $T\ge 1$, we
have $N\ge M$, $\lambda_N\ge 1$, $U_N\le 2N/M$,
$T=N/(MB)\le N/M$, and $T+1\le U_N$.  Hence
\[
    10LA_{T-1}
    =
    \frac5{128}T(T+1)
    \le
    \left(
        \frac{N+M}{M}
    \right)^2
    \le \lambda_N^7 .
\]
By the definition of $\tau$, the inequality
$1-x\le \exp(-x)$ for $x\in[0,1]$, and the choice of $B$,
\[
    \tau
    \le
    \frac{2}{\sqrt\rho}\exp\left(-\frac{\sqrt\rho}{2\sqrt2}B\right)
    \le
    \frac{2}{\sqrt\rho}\left(\frac{\rho}{e\lambda_N}\right)^7
    \le
    \lambda_N^{-7}.
\]
Since $S_T^\eta=A_{T-1}/\beta$ and $\beta\ge1$, this implies
\[
    5LS_T^\eta\tau \le 5L A_{T-1}\tau \le \frac12.
\]
Therefore \cref{lem:distributed-disagreement} applies and gives
\[
    \left(\E\norm{b_t}^2\right)^{1/2}
    \le
    10LS_T^\eta\tau
    \left(\zeta_\star+\frac{\sigma}{\sqrt B}+2LQ_T^{\rm rms}\right)
    \le
    \bar\epsilon,
\]
for every $0\le t<T$, where
\[
    \bar\epsilon
    :=
    10LA_{T-1}\tau(\zeta_\star+\sigma+2LQ_T^{\rm rms}).
\]

We next close the radius dependence in $\bar\epsilon$.  Define the crude
no-bias radius
\[
    C_T
    :=
    2^{10}(T+1)
    \left(
        R+
        \sqrt{\frac{\sigma R}{L\sqrt{MB}}}\,T^{3/4}
    \right).
\]
Applying \cref{lem:expected-distance} to the averaged dynamics with oracle
variance parameter $\sigma^2/(MB)$ and uniform RMS bias bound $\bar\epsilon$
gives
\[
    Q_T^{\rm rms}
    \le
    C_T+2^{10}(T+1)^2A_{T-1}\bar\epsilon .
\]
Moreover,
\[
    T^2(T+1)^4
    \le
    \left(\frac{N+N}{M}\right)^6
    \le \lambda_N^7,
\]
and hence
\[
    c_T
    :=
    20\cdot 2^{10}L^2A_{T-1}^2(T+1)^2\tau
    =
    \frac5{16}T^2(T+1)^4\tau
    \le
    \frac5{16}.
\]
Substituting the definition of $\bar\epsilon$ into the preceding radius bound
gives
\begin{align*}
    Q_T^{\rm rms}
    &\le
    C_T+2^{10}(T+1)^2A_{T-1}\bar\epsilon
    \\
    &=
    C_T
    +2^{10}(T+1)^2A_{T-1}
    \cdot
    10LA_{T-1}\tau(\zeta_\star+\sigma+2LQ_T^{\rm rms})
    \\
    &=
    C_T
    +
    10\cdot 2^{10}L A_{T-1}^2(T+1)^2\tau
    (\zeta_\star+\sigma)
    +
    c_TQ_T^{\rm rms}.
\end{align*}
Since $c_T\le 1/2$,
\[
    Q_T^{\rm rms}
    \le
    2C_T
    +
    20\cdot 2^{10}L A_{T-1}^2(T+1)^2\tau
    (\zeta_\star+\sigma).
\]
Substituting the resulting bound on \(Q_T^{\rm rms}\) back into the
expression \(\zeta_\star+\sigma+2LQ_T^{\rm rms}\) and using
\(2c_T\le 5/8<3\) shows that this quantity is at most a constant
multiple of its bias-free analogue:
\begin{align*}
    \zeta_\star+\sigma+2LQ_T^{\rm rms}
    &\le
    \zeta_\star+\sigma+4LC_T
    +40\cdot 2^{10}L^2A_{T-1}^2(T+1)^2\tau
    (\zeta_\star+\sigma)
    \\
    &=
    (1+2c_T)(\zeta_\star+\sigma)+4LC_T
    \\
    &\le
    4(\zeta_\star+\sigma)+4LC_T
    =
    4(\zeta_\star+\sigma+LC_T).
\end{align*}
Consequently,
\begin{align*}
    \bar\epsilon
    &=
    10LA_{T-1}\tau(\zeta_\star+\sigma+2LQ_T^{\rm rms})
    \\
    &\le
    10LA_{T-1}\tau\cdot 4(\zeta_\star+\sigma+LC_T)
    \\
    &=
    40LA_{T-1}\tau(\zeta_\star+\sigma+LC_T).
\end{align*}

It remains to make this last bound explicit.  Using $B\ge1$, $T\le U_N$, and
$T+1\le U_N$,
\[
    C_T
    \le
    2^{10}U_N
    \left(
        R+
        \sqrt{\frac{\sigma R}{L}}\,U_N^{3/4}
    \right).
\]
Since
\[
    \sqrt{\frac{\sigma}{LR}}
    \le
    1+\frac{\zeta_\star+\sigma}{LR},
\]
and $T^4(T+1)\le U_N^5$, we get
\begin{align*}
    \zeta_\star+\sigma+LC_T
    &\le
    LR\left(
        \frac{\zeta_\star+\sigma}{LR}
        +2^{10}U_N
        +2^{10}U_N^{7/4}\sqrt{\frac{\sigma}{LR}}
    \right)
    \\
    &\le
    \frac{LR}{2^{11}T^4(T+1)}
    2^{11}U_N^5
    \left(
        \frac{\zeta_\star+\sigma}{LR}
        +2^{10}U_N
        +2^{10}U_N^{7/4}\sqrt{\frac{\sigma}{LR}}
    \right)
    \\
    &\le
    \frac{LR}{2^{11}T^4(T+1)}
    \left(2^{16}+2^{27}+2^{28}\right)
    \left(\frac{N}{M}\right)^7
    \left(1+\frac{\zeta_\star+\sigma}{LR}\right)^7
    \\
    &\le
    \frac{LR\lambda_N^7}{2^{11}\sqrt M\,T^4(T+1)}.
\end{align*}
Combining this estimate with $\tau\le\lambda_N^{-7}$ and
$A_{T-1}=T(T+1)/(256L)$ gives
\begin{align*}
    \bar\epsilon
    &\le
    40LA_{T-1}\tau(\zeta_\star+\sigma+LC_T)
    \\
    &\le
    40L\cdot\frac{T(T+1)}{256L}\cdot\lambda_N^{-7}
    \cdot
    \frac{LR\lambda_N^7}{2^{11}\sqrt M\,T^4(T+1)}
    \\
    &=
    \frac{40}{2^{19}}\frac{LR}{\sqrt M\,T^3}
    \le
    2^{-13}\frac{LR}{\sqrt M\,T^3}.
\end{align*}

Thus property~(ii) holds with deterministic RMS bias bounds
\[
    \epsilon_0:=0,
    \qquad
    \epsilon_t:=\bar\epsilon,
    \qquad 1\le t<T.
\]

\medskip
\noindent\textbf{Stage 3: Applying \cref{thm:main} and absorbing the
network error.}
By the recursion, radius condition, centered variance bound, and
properties~(i)--(ii), the hypotheses of \cref{thm:main} hold for the averaged
dynamics with oracle variance parameter $\sigma^2/(MB)$ and the deterministic
RMS bias bounds above.
Let
\[
    \Gamma_T^{\rm net}
    :=
    \frac1{A_{T-1}}\sum_{t=0}^{T-1}a_t\epsilon_t .
\]
Applying \cref{thm:main} to the averaged dynamics and substituting $T=N/(MB)$ gives
\[
    \E[f(\bar x_T)-f(x^\star)]
    \le
    \frac{6\sigma R}{\sqrt N}
    +
    \frac{128LR^2M^2B^2}{N^2}
    +
    \mathcal E_N^{\rm net},
\]
where
\[
    \mathcal E_N^{\rm net}
    :=
    2^{11}
    \left(
        R+\sqrt{\frac{\sigma R}{L}}\,
        \frac{N^{3/4}}{MB}
    \right)
    \frac{N}{MB}\Gamma_T^{\rm net}
    +
    \frac{64}{L}\left(\frac{N}{MB}\right)^4(\Gamma_T^{\rm net})^2.
\]

\textbf{Absorbing the bias terms.}
Since $\epsilon_t\le 2^{-13}LR/T^3$ for every $0\le t<T$, the weighted average satisfies
\[
    \Gamma_T^{\rm net}
    \le
    2^{-13}\frac{LR}{T^3}.
\]

We verify the absorption term by term.  Let
\[
    S_N:=\frac{\sigma R}{\sqrt N},
    \qquad
    D_N:=\frac{LR^2}{T^2}
    =
    \frac{LR^2M^2B^2}{N^2}.
\]
If $\Gamma_T^{\rm net}\le 2^{-13}LR/T^3$, then
\[
    2^{11}RT\Gamma_T^{\rm net}\le \frac14 D_N,
    \qquad
    \frac{64}{L}T^4(\Gamma_T^{\rm net})^2\le 2^{-20}D_N.
\]
For the mixed stochastic-bias term, using
\[
    \sqrt{\frac{\sigma R}{L}}\frac{N^{3/4}}{MB}
    =
    \sqrt{\frac{\sigma R}{L\sqrt{MB}}}\,T^{3/4}
\]
and the same bound on $\Gamma_T^{\rm net}$ gives
\[
    2^{11}
    \sqrt{\frac{\sigma R}{L}}\frac{N^{3/4}}{MB}
    \cdot
    T\Gamma_T^{\rm net}
    \le
    \frac14\sqrt{D_NS_N}
    \le
    \frac18(D_N+S_N).
\]
Combining these estimates gives $\mathcal E_N^{\rm net}\le D_N+S_N$.  Substituting this into the preliminary bound gives the averaged-iterate estimate
\[
    \E[f(\bar x_T)-f(x^\star)]
    \le
    \frac{7\sigma R}{\sqrt N}
    +
    \frac{129LR^2M^2B^2}{N^2}.
\]

\medskip
\noindent\textbf{Stage 4: From the averaged iterate to local outputs.}
It remains to transfer the averaged-iterate guarantee to the local gossiped
outputs.  By \cref{eq:gossiped-output-disagreement} and the definition of
$\bar\epsilon$,
\[
    \left(\E\norm{\widetilde X_T-\bar X_T}_F^2\right)^{1/2}
    \le
    \frac{\sqrt M}{5L}\bar\epsilon,
\]
and hence, for every node $i$,
\[
    \left(\E\norm{\widetilde x_{T,i}-\bar x_T}^2\right)^{1/2}
    \le
    \frac{\sqrt M}{5L}\bar\epsilon
    \le
    2^{-13}\frac{R}{T^3}.
\]
The preceding output-disagreement bound implies
\[
    L\E\norm{\widetilde x_{T,i}-\bar x_T}^2
    \le
    D_N.
\]
For a fixed node $i$, the standard smooth convex inequality
\[
    f(y)-f(x^\star)
    \le
    2(f(x)-f(x^\star))+L\norm{y-x}^2
\]
with $x=\bar x_T$ and $y=\widetilde x_{T,i}$ gives
\[
    \E[f(\widetilde x_{T,i})-f(x^\star)]
    \le
    \frac{14\sigma R}{\sqrt N}
    +
    259D_N.
\]
Finally, since $N\ge M$ gives $\Lambda_N\ge100$ and $\rho\le1$,
\[
    B
    \le
    \frac{20}{\sqrt\rho}\log(\Lambda_N/\rho)+1
    \le
    \frac{21}{\sqrt\rho}\log(\Lambda_N/\rho).
\]
Thus
\[
    D_N
    \le
    \frac{441LR^2M^2}{\rho N^2}
    \log^2(\Lambda_N/\rho),
\]
which gives the claimed bound since $259\cdot441=114219<2^{17}$.
\end{proof}

\section{The strongly convex case through restarting}
\label{app:strongly-convex-restarts}

This appendix records the restart reduction needed for strongly convex
objectives. The only new ingredient is that, after the first epoch, the local
copies used to initialize the next epoch need not be exactly consensual.  

\begin{algorithm}[H]
\caption{Restarted delayed accelerated gossip for strongly convex objectives}
\label{alg:restarted-distributed-gossip}
\begin{algorithmic}[1]
\STATE {\bf Input:} common \(x_0\), number of epochs \(K\), epoch budgets
\((N_s)_{s=0}^{K-1}\), epoch radii \((R_s)_{s=0}^{K-1}\), and the parameter
rules of \cref{thm:distributed-gossip}.
\STATE Set the initial matrix \(U^{(0)}=x_0\mathbf 1^\top\).
\FOR{\(s=0,1,\ldots,K-1\)}
    \STATE Run one warm-started epoch from
    \(X_{-1}=X_0=Z_0=Q_0=\widetilde X_{-1}=U^{(s)}\), using budget \(N_s\) and
    radius parameter \(R_s\).
    \STATE Let \(U^{(s+1)}\) be the matrix of local outputs
    \(\widetilde X^{(s)}\) produced by this epoch.
\ENDFOR
\STATE Each node outputs its column of \(U^{(K)}\).
\end{algorithmic}
\end{algorithm}

\begin{theorem}[Restarted rate for strongly convex objectives]
\label{thm:strongly-convex-restart}
Assume, in addition to the assumptions of \cref{thm:distributed-gossip}, that
\(f\) is \(\mu\)-strongly convex.  Let
\(\Delta_0\ge f(x_0)-f(x^\star)\), fix \(0<\epsilon\le\Delta_0\), and set
\[
    K:=\left\lceil \log_2\frac{\Delta_0}{\epsilon}\right\rceil,
    \qquad
    \Delta_s:=\frac{\Delta_0}{2^s},
    \qquad
    R_s:=\sqrt{\frac{2\Delta_0}{\mu\,2^s}},
    \qquad 0\le s\le K.
\]
Let
\[
    \ell_\epsilon
    :=
    \max\left\{
        60,\,
        \log\left[
            \frac{2^{100}}{\rho\sqrt M}
            \left(1+\frac{\zeta_\star+\sigma}{L\sqrt{\epsilon/\mu}}\right)
            \left(
                1+\frac{\sigma^2}{\mu\epsilon}
                +M\sqrt{\frac{L}{\mu\rho}}
            \right)
        \right]
    \right\},
\]
and choose the epoch budget
\begin{equation}
\label{eq:explicit-restart-budget}
    N_s
    :=
    \left\lceil
        2^{16}\frac{\sigma^2}{\mu\epsilon}
        +
        2^{18}M\sqrt{\frac{L}{\mu\rho}}\,\ell_\epsilon^2
    \right\rceil,
    \qquad 0\le s<K.
\end{equation}
For these budgets define
\[
    \Lambda_s:=
    \frac{100N_s}{\sqrt M}
    \left(1+\frac{\zeta_\star+\sigma}{LR_s}\right),
    \qquad
    B_s:=
    \left\lceil
        \max\left\{1,\frac{20}{\sqrt\rho}
        \log\frac{\Lambda_s}{\rho}\right\}
    \right\rceil.
\]
Assume for simplicity that \(T_s:=N_s/(MB_s)\) is a positive integer for every
\(0\le s<K\).  Run \cref{alg:restarted-distributed-gossip} with these epoch
budgets.  Then the returned local outputs satisfy
$
    \max_{1\le i\le M}
    \E[f(U^{(K)}_{\cdot i})-f(x^\star)]
    \le
    \epsilon
$
with a total sample complexity of
\[
    N
    =
    \sum_{s=0}^{K-1} N_s
    \le
    K\left(
        1+
        2^{16}\frac{\sigma^2}{\mu\epsilon}
        +
        2^{18}M\sqrt{\frac{L}{\mu\rho}}\,\ell_\epsilon^2
    \right)
    =
    \widetilde O\!\left(
        \frac{\sigma^2}{\mu\epsilon}
        +
        M\sqrt{\frac{L}{\mu\rho}}
        \log\frac{\Delta_0}{\epsilon}
    \right).
\]
\end{theorem}

In order to prove \cref{thm:strongly-convex-restart}, we need the following warm-started version of \cref{thm:distributed-gossip}. We defer its proof until the end of this section, after the proof of \cref{thm:strongly-convex-restart}.

\begin{theorem}[Warm-started distributed epoch]
\label{thm:warm-start-distributed-gossip}
Assume the setting of \cref{sec:distributed-problem-setup}.  
Let \(U\in\R^{d\times M}\) be an initialization matrix and write 
\(\bar U:=U\mathbf 1\mathbf 1^\top/M\).  Assume
$
    R^2\ge \E\norm{\bar U_{\cdot,1}-x^\star}^2,
$
define
\[
    \Lambda_N
    :=
    \frac{100 N}{\sqrt M}
    \left(
        1+
        \frac{\zeta_\star+\sigma}{LR}
    \right),
\]
and choose
$
    B
    :=
    \left\lceil
        \max\left\{
            1,\,
            (20/\sqrt\rho)\log(\Lambda_N/\rho)
        \right\}
    \right\rceil.
$
Assume for simplicity that \(T:=N/(MB)\) is a positive integer, and let
\(\eta_t\) be the stepsize prescribed by \cref{thm:distributed-gossip} for
this choice of \(R,N,B,T\).
Let \(S_T^\eta:=\sum_{t=0}^{T-1}\eta_t\),
$
    \tau:=\frac{2}{\sqrt\rho}\left(1-\sqrt{\rho/2}\right)^{B/2},
$
and assume that the RMS column disagreement of the initialization satisfies
\begin{equation}
\label{eq:warm-start-disagreement-small}
    \left(\E\norm{U-\bar U}_F^2\right)^{1/2}
    \le
    2^{-16}\frac{R}{T},
    \qquad
    \tau(2+3LS_T^\eta)
    \left(\E\norm{U-\bar U}_F^2\right)^{1/2}
    \le
    2^{-24}\frac{R}{T^3}.
\end{equation}
Run \cref{alg:distributed-delayed-gossip} with the same parameter choices as in
\cref{thm:distributed-gossip}, except that the initialization line is replaced by
\[
    X_{-1}=X_0=Z_0=Q_0=\widetilde X_{-1}=U .
\]
Then, for every node \(i\),
\[
    \E[f(\widetilde x_{T,i})-f(x^\star)]
    \le
    \frac{20\sigma R}{\sqrt N}
    +
    \frac{2^{21}LR^2M^2}{\rho N^2}
    \log^2 \frac{\Lambda_N}{\rho}.
\]
Moreover, the final output disagreement satisfies
\begin{equation}
\label{eq:warm-start-output-disagreement}
    \left(\E\norm{\widetilde X_T-\bar X_T}_F^2\right)^{1/2}
    \le
    2^{-11}\frac{R}{T^3}.
\end{equation}
\end{theorem}

We are ready to prove the convergence for the restarted method.

\begin{proof}[Proof of \cref{thm:strongly-convex-restart}]
For the chosen budgets, let
\[
    \tau_s:=\frac{2}{\sqrt\rho}
    \left(1-\sqrt{\rho/2}\right)^{B_s/2},
    \qquad
    S_s^\eta:=\sum_{t=0}^{T_s-1}\eta_t,
\]
where \(\eta_t\) is the epoch stepsize used with radius \(R_s\).  Set
\(d_0:=0\), and for \(1\le s<K\) put
\[
    d_s:=2^{-11}\frac{R_{s-1}}{T_{s-1}^3},
    \qquad 1\le s<K.
\]
We now check that the explicit budget is large enough for the two restart
requirements: the inherited disagreement is small enough for the next
warm-started epoch, and the current epoch reduces the error by a factor two.
For \(0\le s<K\), \(R_s\ge \sqrt{\epsilon/\mu}\), and
\cref{eq:explicit-restart-budget} gives
\[
    N_s
    \le
    2^{19}\ell_\epsilon^2
    \left(
        1+\frac{\sigma^2}{\mu\epsilon}
        +M\sqrt{\frac{L}{\mu\rho}}
    \right).
\]
By the definition of \(\ell_\epsilon\) and the slack factor \(2^{100}\),
\[
    \log\frac{\Lambda_s}{\rho}
    \le
    2\ell_\epsilon.
\]
Indeed, after taking logarithms, the only extra contribution beyond the
quantity inside \(\ell_\epsilon\) is
\(2\log\ell_\epsilon+\log(100\cdot2^{19})\), which is at most
\(\ell_\epsilon\) because \(\ell_\epsilon\ge60\).
Consequently
\[
    B_s
    \le
    \frac{41}{\sqrt\rho}\ell_\epsilon,
    \qquad
    T_s
    =
    \frac{N_s}{MB_s}
    \ge
    \frac{2^{18}}{41}\sqrt{\frac{L}{\mu}}\ell_\epsilon
    \ge
    16.
\]
Here we used the standard consequence \(\mu\le L\) of \(\mu\)-strong
convexity and \(L\)-smoothness.
Since \(N_s\) is independent of \(s\) and \(R_s\) is decreasing, \(B_s\) is
nondecreasing and therefore \(T_s\) is nonincreasing.

The inherited disagreement then satisfies
\begin{equation}
\label{eq:restart-explicit-disagreement-small}
    d_s
    \le
    2^{-16}\frac{R_s}{T_s},
    \qquad
    \tau_s(2+3LS_s^\eta)d_s
    \le
    2^{-24}\frac{R_s}{T_s^3},
\end{equation}
for every \(1\le s<K\).  Indeed, \(R_{s-1}=\sqrt2R_s\) and
\(T_{s-1}\ge T_s\), so the first inequality follows from \(T_s\ge16\).  For
the second, put
\[
    \lambda_s:=
    \frac{32N_s}{\sqrt M}
    \left(1+\frac{\zeta_\star+\sigma}{LR_s}\right).
\]
As in the proof of \cref{thm:distributed-gossip}, \(\Lambda_s\ge e\lambda_s\)
and the choice of \(B_s\) imply \(\tau_s\le\lambda_s^{-7}\), while
\(T_s^2(T_s+1)^4\le\lambda_s^7\).  Since
\(LS_s^\eta\le LA_{T_s-1}=T_s(T_s+1)/256\),
\[
    \tau_s(2+3LS_s^\eta)
    \le
    \frac{4}{(T_s+1)^4}
    \le
    2^{-14},
\]
and the second inequality in \cref{eq:restart-explicit-disagreement-small}
follows from the first display defining \(d_s\).

The same explicit budget also gives the epoch-accuracy bound
\begin{equation}
\label{eq:epoch-halving-condition}
    \frac{20\sigma R_s}{\sqrt{N_s}}
    +
    \frac{2^{21}LR_s^2M^2}{\rho N_s^2}
    \log^2\frac{\Lambda_s}{\rho}
    \le
    \frac{\Delta_0}{2^{s+1}},
\end{equation}
for every \(0\le s<K\).  The stochastic term is at most
\((20\sqrt2/2^8)\Delta_s\le\Delta_s/4\), because
\(\Delta_s\ge\epsilon\).  The network term is at most
\[
    \frac{2^{23}LR_s^2M^2}{\rho N_s^2}\ell_\epsilon^2
    \le
    2^{-12}\frac{\Delta_s}{\ell_\epsilon^2}
    \le
    \frac{\Delta_s}{4}.
\]
We prove by induction that, at the beginning of epoch \(s\),
\[
    \E[f(\bar U^{(s)}_{\cdot 1})-f(x^\star)]\le \Delta_s,
    \qquad
    \E\norm{\bar U^{(s)}_{\cdot 1}-x^\star}^2\le R_s^2,
    \qquad
    \left(\E\norm{U^{(s)}-\bar U^{(s)}}_F^2\right)^{1/2}\le d_s.
\]
The claim is true at \(s=0\) by the definition of \(\Delta_0\) and strong
convexity, and by \(d_0=0\).  Suppose it holds at epoch \(s\).  The inherited
disagreement bound together with \cref{eq:restart-explicit-disagreement-small}
verifies the warm-start smallness assumption of
\cref{thm:warm-start-distributed-gossip} for the initialization \(U^{(s)}\).
Applying that theorem with radius \(R_s\), and then using
\cref{eq:epoch-halving-condition}, gives, for every node \(i\),
\[
    \E[f(U^{(s+1)}_{\cdot i})-f(x^\star)]
    \le
    \Delta_{s+1}.
\]
By convexity, the network average at the next restart also satisfies
\[
    \E[f(\bar U^{(s+1)}_{\cdot 1})-f(x^\star)]
    \le
    \frac1M\sum_{i=1}^M
    \E[f(U^{(s+1)}_{\cdot i})-f(x^\star)]
    \le
    \Delta_{s+1}.
\]
Strong convexity then gives
\[
    \E\norm{\bar U^{(s+1)}_{\cdot 1}-x^\star}^2
    \le
    \frac{2\Delta_{s+1}}{\mu}
    =
    R_{s+1}^2.
\]
If \(s<K-1\), then the output-disagreement estimate
\cref{eq:warm-start-output-disagreement} gives
\[
    \left(\E\norm{U^{(s+1)}-\bar U^{(s+1)}}_F^2\right)^{1/2}
    \le
    2^{-11}\frac{R_s}{T_s^3}
    =
    d_{s+1},
\]
which closes the induction.

After \(K\) epochs, \(\Delta_K\le\epsilon\), proving the accuracy claim.  The
sample-complexity bound is the explicit sum of the constant epoch budgets in
\cref{eq:explicit-restart-budget}; the factors \(K\) and
\(\ell_\epsilon^2\) are logarithmic in the target accuracy and problem
parameters, giving the stated \(\widetilde O\) form.
\end{proof}

\subsection{Proof of \texorpdfstring{\cref{thm:warm-start-distributed-gossip}}{Theorem 6}}

\begin{proof}[\unskip\nopunct]
Let
\[
    D_U^{\rm rms}:=
    \left(\E\norm{U-\bar U}_F^2\right)^{1/2}.
\]
For the warm initialization above,
\[
    D_{\rm init}^{\rm rms}=2D_U^{\rm rms},
    \qquad
    Q_{\rm init}^{\rm rms}=D_U^{\rm rms},
    \qquad
    \Delta_{\rm init}^{\rm rms}=(2+3LS_T^\eta)D_U^{\rm rms}.
\]
The direct effect of the initial query occurs only at \(t=0\).  Indeed,
\[
    \left(\E\norm{b_0}^2\right)^{1/2}
    \le
    \frac{L}{\sqrt M}D_U^{\rm rms},
\]
whereas for \(t\ge1\) the query is formed from already-gossiped primal copies.
Writing \(D_T^{\rm rms}:=\max_{0\le t\le T}
(\E\norm{X_t-\bar X_t}_F^2)^{1/2}\), the accelerated-gossip contraction gives
\[
    \left(\E\norm{Q_t-\bar Q_t}_F^2\right)^{1/2}
    \le
    5\tau D_T^{\rm rms},
    \qquad 1\le t<T.
\]
Repeating the proof of \cref{lem:distributed-disagreement} with this sharper
query-disagreement bound, and using the same verification of
\(5LS_T^\eta\tau\le1/2\) as in the proof of
\cref{thm:distributed-gossip}, gives, for \(t\ge1\),
\[
    \left(\E\norm{b_t}^2\right)^{1/2}
    \le
    10LS_T^\eta\tau
    \left(\zeta_\star+\frac{\sigma}{\sqrt B}+2LQ_T^{\rm rms}\right)
    +
    \frac{10L\tau\Delta_{\rm init}^{\rm rms}}{\sqrt M}.
\]
The assumption \cref{eq:warm-start-disagreement-small} implies
\[
    \frac{a_0}{A_{T-1}}\frac{L}{\sqrt M}D_U^{\rm rms}
    \le
    2^{-15}\frac{LR}{T^3},
    \qquad
    \frac{10L\tau\Delta_{\rm init}^{\rm rms}}{\sqrt M}
    \le
    2^{-20}\frac{LR}{T^3},
    \qquad
    2\tau\Delta_{\rm init}^{\rm rms}
    \le
    2^{-23}\frac{R}{T^3}.
\]
Thus the weighted RMS bias budget in \cref{thm:main} is enlarged by at most
\(2^{-14}LR/T^3\) compared with the consensual-start proof of
\cref{thm:distributed-gossip}.  Repeating the same radius-closing and
bias-absorption argument with this larger bias budget gives
\[
    \E[f(\bar x_T)-f(x^\star)]
    \le
    \frac{10\sigma R}{\sqrt N}
    +
    \frac{2^{20}LR^2M^2}{\rho N^2}
    \log^2 \frac{\Lambda_N}{\rho}.
\]
Here we use the RMS-start version of \cref{thm:main}, which follows from the
same proof by taking expectations in the initial term
\(\frac{\beta}{2}\norm{x^\star-x_0}^2\); hence it is enough that
\(R^2\ge\E\norm{\bar U_{\cdot,1}-x^\star}^2\).

The output-disagreement part of \cref{lem:distributed-disagreement} gives the
usual term from the proof of \cref{thm:distributed-gossip}, plus an additive
term of at most \(2^{-23}R/T^3\).  Hence
\[
    \left(\E\norm{\widetilde X_T-\bar X_T}_F^2\right)^{1/2}
    \le
    2^{-11}\frac{R}{T^3},
\]
after enlarging constants.  The same smoothness comparison between
\(\widetilde x_{T,i}\) and \(\bar x_T\) used at the end of
\cref{thm:distributed-gossip}'s proof then yields the stated local bound.
\end{proof}

\section{A Lower Bound for Linear-Span Decentralized Methods}
\label{app:linear-span-lower-bound}

This appendix proves the lower bound stated in \cref{thm:linear-span-lower-bound} for linear-span 
decentralized methods.  The linear-span restriction is important---it rules out methods that guess 
arbitrary directions unrelated to the gradients and messages seen so far---but it is a standard setting in which zero-chain lower-bound constructions are established \citep[e.g.,][]{nesterov2018lectures}.

Consider the path graph on \(M\) nodes, with \(M\ge 6\).  Let
\(P_{\rm path}=I-\frac14\mathsf L_{\rm path}\), where
\(\mathsf L_{\rm path}\) is the unnormalized Laplacian of the path.  This is a
symmetric doubly stochastic gossip matrix supported on the path, and its
spectral gap satisfies, by the standard spectrum of the path Laplacian
\citep{chung1997spectral},
\[
    \rho
    =
    \frac12\left(1-\cos\frac{\pi}{M}\right)
    =
    \Theta(M^{-2}).
\]
A linear-span decentralized first-order method is allowed, in each communication
round, to perform arbitrary local linear algebra on vectors already held by a
node, query the local gradient at points in the span of those vectors, and send
arbitrary vectors from this span to neighboring nodes.  Equivalently, if
\(\mathcal S_i(t)\) is the subspace available to node \(i\) after \(t\) gossip
communication rounds, then \(\mathcal S_i(0)=\{0\}\), and
\(\mathcal S_i(t+1)\) is contained in the span of \(\mathcal S_i(t)\), the
neighboring subspaces
\(\mathcal S_j(t)\), \(\{i,j\}\in E\), and local gradients evaluated at points
in that span.  All nodes are
initialized at the origin.  The proof below gives the method the additional
power of free local span closure inside each block, so the lower bound also
applies to the one-query-per-round linear-span model.

\begin{proof}[Proof of \cref{thm:linear-span-lower-bound}]
Let \(m:=\lfloor M/3\rfloor\), and partition the path into
\[
    V_L:=\{1,\ldots,m\},
    \qquad
    V_R:=\{M-m+1,\ldots,M\},
\]
and the remaining middle block.  The graph distance between \(V_L\) and
\(V_R\) is
\[
    \Delta:=M-2m+1=\Theta(M)=\Theta(\rho^{-1/2}).
\]
Set
\[
    s:=2+\left\lfloor \frac{S}{\Delta}\right\rfloor,
    \qquad
    d=2(s+1),
    \qquad
    a:=\frac{R}{\sqrt d}.
\]
For \(z\in\R^d\), define
\[
    H_L(z)
    :=
    z_1^2-2z_1
    +
    \sum_{\substack{1\le r<d\\ r\ {\rm odd}}}
    (z_r-z_{r+1})^2,
    \qquad
    H_R(z)
    :=
    \sum_{\substack{1\le r<d\\ r\ {\rm even}}}
    (z_r-z_{r+1})^2,
\]
and \(H:=H_L+H_R\).  Fix a sufficiently small universal constant
\(c_0>0\), for instance \(c_0=2^{-6}\), and define the local objectives by
\[
    f_i(x)
    :=
    c_0L\frac{M}{m}a^2H_L(x/a),
    \qquad i\in V_L,
\]
\[
    f_i(x)
    :=
    c_0L\frac{M}{m}a^2H_R(x/a),
    \qquad i\in V_R,
\]
and \(f_i(x):=0\) for all middle-block nodes.  Since \(m\ge M/4\) for
\(M\ge6\), and the Hessians of \(H_L\) and \(H_R\) have operator norm bounded
by an absolute constant, the choice of \(c_0\) ensures that each \(f_i\) is
convex and \(L\)-smooth.  Averaging over the nodes gives
\[
    f(x)=c_0La^2H(x/a).
\]
The vector \(z^\star=\mathbf 1\in\R^d\) minimizes both \(H_L\) and \(H_R\),
and hence \(x^\star:=a\mathbf 1\) minimizes \(f\).  Moreover,
\(\norm{x^\star}=a\sqrt d=R\).

We next record the zero-chain property.  Let
\[
    E_r:=\operatorname{span}\{e_1,\ldots,e_r\},
    \qquad E_0:=\{0\}.
\]
The odd-even split of the coordinate couplings is what turns the usual
zero-chain construction into a communication lower bound.  
If \(z\in E_r\), then a gradient of \(H_L\) can introduce the new coordinate
\(e_{r+1}\) only when \(r\) is odd, while a gradient of \(H_R\) can introduce
the new coordinate \(e_{r+1}\) only when \(r\) is even.  This follows directly
from the fact that \(H_L\) contains only the odd couplings
\((z_r-z_{r+1})^2\), whereas \(H_R\) contains only the even couplings.  
% The
% linear term in \(H_L\) starts the chain, and the left block can generate the
% first two coordinates without receiving any information from the right block.
The left objective
contains the linear term and the couplings \((1,2),(3,4),\ldots\), while the
right objective contains the couplings \((2,3),(4,5),\ldots\).  Thus the left
block can start the chain and reveal coordinate \(2\), but the next coupling
needed to reveal coordinate \(3\) is available only in the right block.  After
the right block reveals coordinate \(3\), the next coupling needed to reveal
coordinate \(4\) is again available only in the left block, and so on.

For \(r\ge2\), coordinate \(r+1\) must be generated in \(V_R\) when \(r\) is
even and in \(V_L\) when \(r\) is odd, from a vector whose \(r\)-th coordinate
is already nonzero.  These generating blocks alternate.  Since a message can
move by at most one edge per communication round, information needs at least
\(\Delta\) rounds to travel from one generating block to the other.  Therefore,
even in the strengthened model that allows arbitrary local span closure inside
each block, no coordinate beyond \(E_s\) can appear anywhere in the network
after \(S\) gossip communication rounds.  Hence every local output, and also
their average \(\bar x_S\), belongs to \(E_s\).

It remains to lower bound the objective gap of vectors in \(E_s\).  For
\(z=x/a\), the identity
\[
    H(z)-H(\mathbf 1)
    =
    (z_1-1)^2+\sum_{r=1}^{d-1}(z_r-z_{r+1})^2
\]
holds.  If \(x\in E_s\), then \(z\in E_s\), so \(z_{s+1}=0\).  By
Cauchy-Schwarz inequality,
\[
    1
    =
    (1-z_1)+(z_1-z_2)+\cdots+(z_s-z_{s+1})
\]
implies
\[
    (z_1-1)^2+\sum_{r=1}^{s}(z_r-z_{r+1})^2
    \ge
    \frac1{s+1}.
\]
Consequently, every \(x\in E_s\) satisfies
\[
    f(x)-f(x^\star)
    =
    c_0La^2\bigl(H(x/a)-H(\mathbf 1)\bigr)
    \ge
    \frac{c_0LR^2}{d(s+1)}
    =
    \frac{c_0LR^2}{2(s+1)^2}.
\]
Since \(s+1\le 3+S/\Delta\) and \(\Delta=\Theta(\rho^{-1/2})\), the last
display is bounded below by
\[
    cLR^2
    \min\left\{1,\frac1{\rho S^2}\right\}
\]
for another universal constant \(c>0\).  Applying this to
\(x=\bar x_S\in E_s\) proves the theorem.
\end{proof}

\section{Numerical experiments}
\label{app:experiments}

In this section, we provide the full details of our experimental evaluation (\cref{sec:experiments}).  We follow the synthetic least-squares benchmark of \citet{eisen2025enhancing}.  The experiments address two questions.  First, how does \dadq{} compare with D-SGD and DAT-SGD~\citep{eisen2025enhancing} under the same sample and communication budget?  Second, how sensitive is \dadq{} to the block length $B$?  This is the only parameter of \cref{alg:distributed-delayed-gossip} without a counterpart in previous methods, and it simultaneously determines the local minibatch size, the number of accelerated gossip rounds per macro step and, through $T=N/(MB)-1$, the number of updates.  All results are averaged over three random seeds.

\paragraph{Problem.}
The local objective of worker $i$ is $f_i(x)=\frac12\norm{A_ix-b_i}^2$, where $A_i\in\R^{d\times d}$ has i.i.d.\ standard normal entries and $b_i=A_i(x^\sharp-\delta_i)$.  The vector $x^\sharp\sim\mathcal N(0,\frac1d I_d)$ is sampled once per configuration, and $\delta_i\sim\mathcal N(0,\frac{\zeta^2}{d}I_d)$ introduces heterogeneity across workers.  A stochastic gradient query at $x$ returns $\nabla f_i(x)+\xi$ with $\xi\sim\mathcal N(0,\frac{\sigma^2}{d}I_d)$.  We set $d=50$, $\sigma^2,\zeta^2\in\{1,10\}$ and $M\in\{4,9,25,49,100\}$.  All workers start from a common point $x_0\sim\mathcal N(0,I_d)$, as \cref{alg:distributed-delayed-gossip} requires, so the initial error is about $50$.  Following \citet{eisen2025enhancing}, we report the mean squared distance $\frac1M\sum_{i=1}^M\norm{x_i-x^\star}^2$ between the local outputs $x_i$ and the minimizer $x^\star$ of $f$.

\paragraph{Networks.}
We use the ring and the torus of \citet{eisen2025enhancing}.  On the ring, each worker averages itself and its two neighbors with weight $\frac13$; on the two-dimensional torus, itself and its four neighbors with weight $\frac15$ (for $M=4$, the torus reduces to the ring).  Both gossip matrices are fixed, symmetric and doubly stochastic; their spectral gaps are listed in \cref{tab:exp-block-length}.

\paragraph{Budget and methods.}
Every method runs for $10^5$ communication rounds, and in each round every worker draws one stochastic gradient, so that $N=10^5M$.  D-SGD returns its last iterate and DAT-SGD its averaged iterate, as in \citet{eisen2025enhancing}; DAT-SGD exchanges two vectors per round, whereas D-SGD and \dadq{} exchange one.  With block length $B$, \dadq{} performs $T=10^5/B-1$ updates: the final block of $B$ gradients is charged to the budget but, since $\eta_T=\theta_T=0$, does not affect the output $\widetilde x_{T,i}$.  We run \dadq{} with two parameter schedules, both normalized so that the effective primal step $\theta_t\eta_t$ equals the tuned stepsize $\eta$ during updates.
\begin{itemize}[leftmargin=*]
\item \emph{Theorem schedule}: the momentum and averaging schedules of \cref{thm:distributed-gossip}, with tuned stepsize and block length: $\alpha_{\rm g}=\frac12$, $\beta_{\rm g}=(1-\sqrt{\rho/2})/(1+\sqrt{\rho/2})$, $\theta_t=2/(t+2)$, $m_t=(t-1)/(t+2)$ and $\eta_t=\eta(t+2)/2$.
\item \emph{Constant momenta}: a practical variant with $\theta_t\equiv0.1$, $m_t\equiv0.9$ and $\eta_t\equiv\eta/0.1$, and with $\alpha_{\rm g}=\frac12$, $\beta_{\rm g}=0.9$ on every graph.  The momenta satisfy the same relation $m_t=\theta_t(1-\theta_{t-1})/\theta_{t-1}$ as in \cref{thm:distributed-gossip}.
\end{itemize}

\paragraph{Tuning.}
As in \citet{eisen2025enhancing}, for every method, topology, $M$, $\sigma^2$ and $\zeta^2$, we select the stepsize $\eta\in\{10^{-4},5\cdot10^{-4},10^{-3},5\cdot10^{-3},10^{-2},5\cdot10^{-2},10^{-1}\}$ with the lowest final error.  For \dadq{}, we select $B$ jointly with $\eta$ from the twenty values
\[
    B\in\{1,2,4,5,10,16,20,25,50,100,200,250,400,500,625,1000,1250,2000,2500,5000\}.
\]
This gives \dadq{} a larger tuning budget than the baselines. We select hyperparameters by minimizing mean final error over three random seeds and report the mean and sample standard deviation over those same seeds.

\paragraph{Compute.}
The experiments ran on one computer with 10 CPU cores and 32 GB of RAM. The block-length grid took approximately two hours of wall-clock time, and the comparison with both baselines took approximately 1.5 hours on four cores.

\subsection{Comparison with D-SGD and DAT-SGD}
\label{app:experiments-comparison}

\Cref{fig:exp-comparison} shows both schedules of \dadq{}, together with D-SGD and DAT-SGD, for every pair $(\sigma^2,\zeta^2)\in\{1,10\}^2$, in the layout of \citet[Figure~1]{eisen2025enhancing}.  \Cref{tab:exp-comparison} records the case $\sigma^2=\zeta^2=1$, with the block length selected for \dadq{}.

\begin{figure}[t]
\centering
\includegraphics[width=\linewidth]{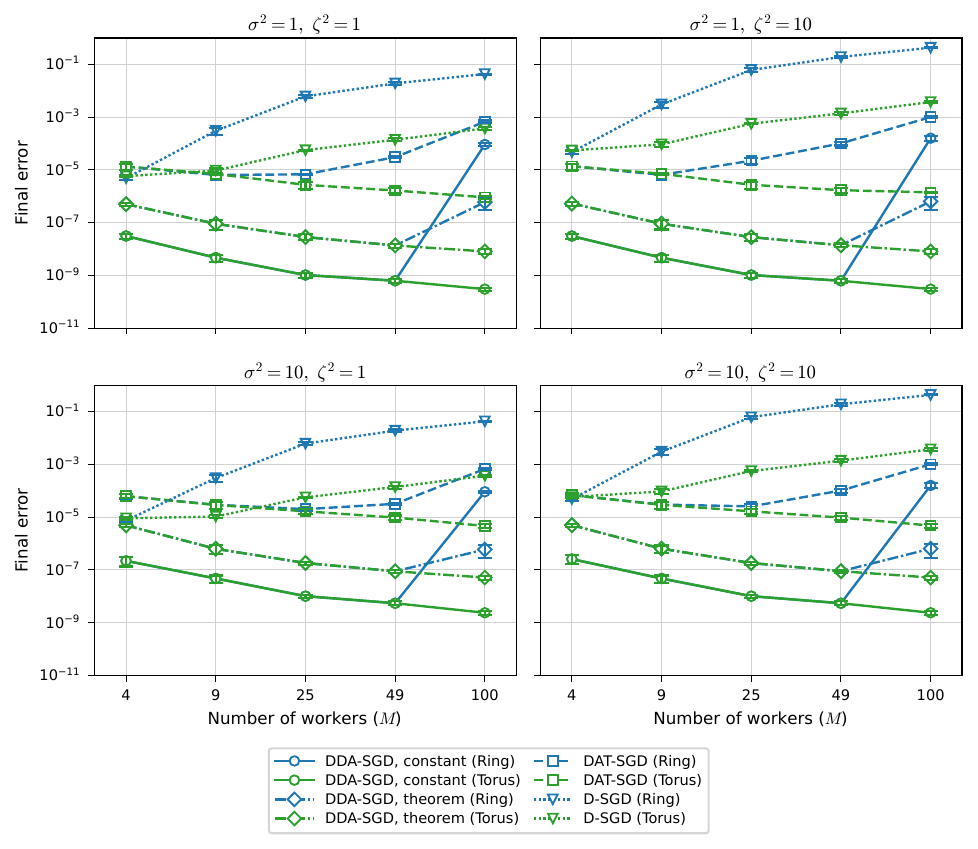}
\caption{Final error against the number of workers for each pair $(\sigma^2,\zeta^2)\in\{1,10\}^2$, in the layout of \citet[Figure~1]{eisen2025enhancing}.
Color denotes the topology and line style the method.
The plotted quantity is $\frac1M\sum_{i=1}^M\norm{x_i-x^\star}^2$, with $x_i$ the local output of the method.
Error bars show the sample standard deviation across three random seeds.
\Cref{tab:exp-comparison} lists the case $\sigma^2=\zeta^2=1$.}
\label{fig:exp-comparison}
\end{figure}

\begin{table}[t]
\centering
\small
\begin{tabular}{llllll}
\toprule
Graph & $M$ & D-SGD & DAT-SGD & \dadq{} (constant) & \dadq{} (theorem) \\
\midrule
Ring & 4 & $4.8(0.6)\cdot 10^{-6}$ & $1.3(0.3)\cdot 10^{-5}$ & $\mathbf{3.1(0.6)\cdot 10^{-8}}$ (500) & $5.0(0.6)\cdot 10^{-7}$ (200) \\
 & 9 & $3.0(0.8)\cdot 10^{-4}$ & $6.4(0.2)\cdot 10^{-6}$ & $\mathbf{4.7(1.4)\cdot 10^{-9}}$ (250) & $9.1(3.6)\cdot 10^{-8}$ (250) \\
 & 25 & $6.1(0.9)\cdot 10^{-3}$ & $6.8(0.5)\cdot 10^{-6}$ & $\mathbf{1.0(0.2)\cdot 10^{-9}}$ (400) & $2.8(0.7)\cdot 10^{-8}$ (400) \\
 & 49 & $1.9(0.2)\cdot 10^{-2}$ & $3.0(0.3)\cdot 10^{-5}$ & $\mathbf{6.3(1.3)\cdot 10^{-10}}$ (400) & $1.4(0.2)\cdot 10^{-8}$ (400) \\
 & 100 & $4.3(0.2)\cdot 10^{-2}$ & $6.7(0.8)\cdot 10^{-4}$ & $9.2(1.0)\cdot 10^{-5}$ (1000) & $\mathbf{6.0(3.0)\cdot 10^{-7}}$ (625) \\
\midrule
Torus & 4 & $5.8(0.4)\cdot 10^{-6}$ & $1.3(0.3)\cdot 10^{-5}$ & $\mathbf{3.1(0.6)\cdot 10^{-8}}$ (500) & $5.0(0.6)\cdot 10^{-7}$ (200) \\
 & 9 & $9.4(1.8)\cdot 10^{-6}$ & $7.0(0.6)\cdot 10^{-6}$ & $\mathbf{4.7(1.4)\cdot 10^{-9}}$ (250) & $9.1(3.6)\cdot 10^{-8}$ (250) \\
 & 25 & $5.6(0.3)\cdot 10^{-5}$ & $2.7(0.8)\cdot 10^{-6}$ & $\mathbf{1.0(0.2)\cdot 10^{-9}}$ (400) & $2.8(0.7)\cdot 10^{-8}$ (400) \\
 & 49 & $1.4(0.2)\cdot 10^{-4}$ & $1.7(0.3)\cdot 10^{-6}$ & $\mathbf{6.3(1.3)\cdot 10^{-10}}$ (400) & $1.4(0.2)\cdot 10^{-8}$ (400) \\
 & 100 & $3.7(0.4)\cdot 10^{-4}$ & $9.1(0.9)\cdot 10^{-7}$ & $\mathbf{3.0(0.4)\cdot 10^{-10}}$ (400) & $8.1(1.6)\cdot 10^{-9}$ (200) \\
\bottomrule
\end{tabular}
\setlength{\abovecaptionskip}{7pt}
\setlength{\belowcaptionskip}{0pt}
\caption{Final error for $\sigma^2=\zeta^2=1$, the top-left panel of \cref{fig:exp-comparison}.  The parenthetical number on each error is the standard deviation across three seeds, in the same units.  The integer after a \dadq{} entry is the selected block length.  The lowest error in each row is in bold.}
\label{tab:exp-comparison}
\end{table}

\paragraph{Optimization performance.}
\Cref{fig:exp-comparison} shows that \dadq{} achieves lower error than both baselines across all tested ring and torus configurations.  This holds for both momentum schedules.  With constant momenta, the final error is at least $6$ times smaller than that of DAT-SGD and more than $30$ times smaller than that of D-SGD.  With the theorem-based schedule, the improvement over DAT-SGD is at least $12$-fold.  More precisely, constant momenta improve on DAT-SGD by a factor between $6$ and $1.6\cdot10^{5}$ (median $1.6\cdot10^{3}$), and the theorem schedule improves on DAT-SGD by a factor between $12.7$ and $7.3\cdot10^{3}$ (median $94$).

\paragraph{Parallelism.}
At a fixed per-worker budget, \dadq{} improves on the torus through $M=100$ in all four panels of \cref{fig:exp-comparison}; at $\sigma^2=\zeta^2=1$, error falls a hundredfold from $M=4$.  On the ring, it improves through $M=49$ before error rises at $M=100$, consistent with network error overtaking statistical error.  The theorem-based schedule mitigates this rise, from $1.4\cdot10^{-8}$ to $6.0\cdot10^{-7}$ (\cref{tab:exp-comparison}).  D-SGD deteriorates on both networks from the smallest $M$, especially at $\zeta^2=10$.  DAT-SGD improves on the torus, while its lowest ring error occurs at $M=9$ or $M=25$, depending on the noise setting.  Thus, \dadq{} benefits from additional workers over the widest range on the ring.

\paragraph{Noise versus heterogeneity.}
The four panels separate the two error sources.  Passing from the left column to the right multiplies $\zeta^2$ by ten, and passing from the top row to the bottom multiplies $\sigma^2$ by ten.  The curves of \dadq{} with constant momenta barely move with $\zeta^2$ (median ratio $1.00$; the largest ratio, $1.74$, occurs on the $100$-worker ring) and rise by a median factor of $8.7$ with $\sigma^2$.  Thus, in most tested configurations, these trends are consistent with variance-dominated error and weak sensitivity to heterogeneity, as expected when the residual disagreement is negligible.  D-SGD exhibits the opposite pattern (median ratios $1.0$ in $\sigma^2$ and $10$ in $\zeta^2$), so its error is dominated by heterogeneity-induced disagreement, while DAT-SGD lies in between (median ratios $4.6$ and $1.03$).

\paragraph{Summary.}
On the ring and the torus, \dadq{} is consistently and substantially more accurate than D-SGD and DAT-SGD, and benefits from additional workers over the widest range of $M$, provided that $B$ is long enough for the gossip block to reach consensus, yet short enough to leave sufficiently many updates.  The theorem schedule is robust and excels where communication is the bottleneck, whereas constant momenta are more accurate elsewhere.

\subsection{Sensitivity to the block length}
\label{app:experiments-block}

\Cref{fig:exp-block-length} shows the final error of \dadq{} as a function of $B$ for both schedules, with $\eta$ tuned separately for each $B$, and \cref{tab:exp-block-length} lists the minimizing block lengths.  We fix $\sigma^2=\zeta^2=1$; the other variance levels shift the curves vertically but barely move their minimizers.

\begin{figure}[t]
\centering
\includegraphics[width=\linewidth,trim=0 5bp 0 6bp,clip]{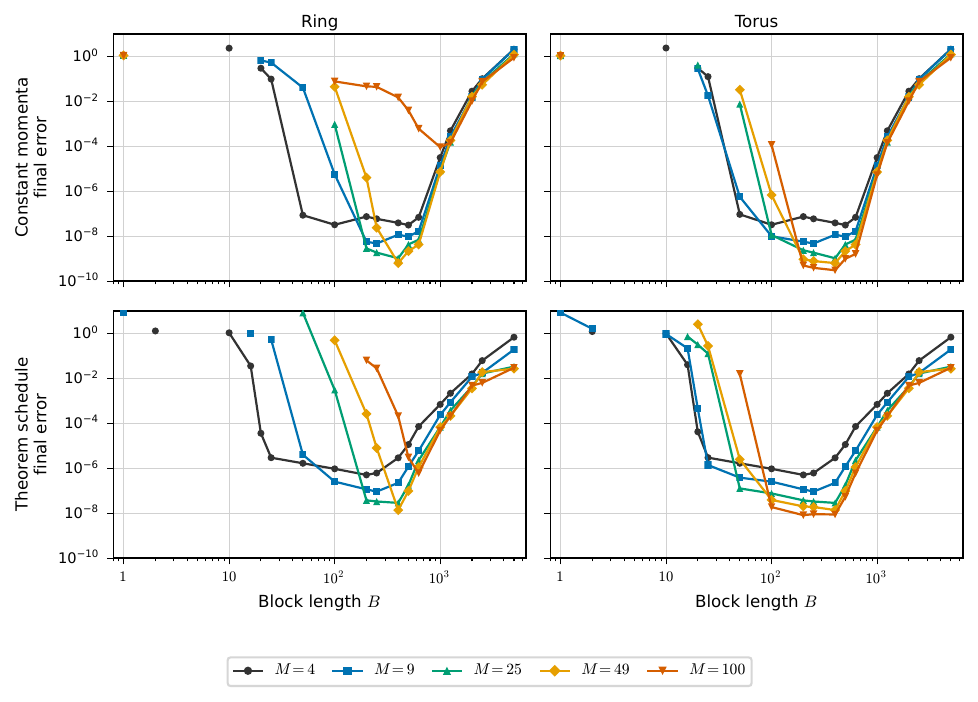}
\caption{Final error of \dadq{} against the block length $B$ for $\sigma^2=\zeta^2=1$.
The top row uses constant momenta and the bottom row the schedule of \cref{thm:distributed-gossip}; color denotes the number of workers.
Both axes are logarithmic, and the stepsize is tuned separately at each $B$.
Runs that diverged, or whose error exceeds $10$, are omitted, which opens gaps in the curves.}
\label{fig:exp-block-length}
\end{figure}

\begin{table}[!h]
\centering
\small
\begin{tabular}{llrrrrrr}
\toprule
& & & & \multicolumn{2}{c}{Constant momenta} & \multicolumn{2}{c}{Theorem schedule} \\
\cmidrule(lr){5-6}\cmidrule(lr){7-8}
Graph & $M$ & $\rho$ & $B_{\rm thm}$ & $B^\star$ & range & $B^\star$ & range \\
\midrule
Ring & 4 & $0.67$ & 422 & 500 & 50--625 & 200 & 25--400 \\
 & 9 & $0.16$ & 966 & 250 & 200--625 & 250 & 100--400 \\
 & 25 & $0.021$ & 2985 & 400 & 200--625 & 400 & 200--500 \\
 & 49 & $0.0055$ & 6292 & 400 & 400--625 & 400 & 400--500 \\
 & 100 & $0.0013$ & 13816 & 1000 & 625--1250 & 625 & 500--625 \\
\midrule
Torus & 4 & $0.67$ & 422 & 500 & 50--625 & 200 & 25--400 \\
 & 9 & $0.60$ & 458 & 250 & 100--625 & 250 & 50--400 \\
 & 25 & $0.28$ & 724 & 400 & 200--625 & 400 & 50--500 \\
 & 49 & $0.15$ & 1029 & 400 & 200--625 & 400 & 100--500 \\
 & 100 & $0.076$ & 1520 & 400 & 200--625 & 200 & 100--500 \\
\bottomrule
\end{tabular}
\setlength{\abovecaptionskip}{7pt}
\setlength{\belowcaptionskip}{-12pt}
\caption{Block-length selection for \dadq{} ($\sigma^2=\zeta^2=1$).  $\rho$ is the spectral gap.  $B_{\rm thm}$ is the block length prescribed by \cref{thm:distributed-gossip}, with the factor $1+(\zeta_\star+\sigma)/(LR)$ in $\Lambda_N$ set to one.  $B^\star$ is the minimizing block length, and the range lists the block lengths of the grid whose error is within a factor of ten of the best.}
\label{tab:exp-block-length}
\end{table}

\paragraph{Three regimes.}
On the ring and the torus, the curves in \cref{fig:exp-block-length} share one shape under both schedules.  For $B\le5$, no configuration attains an error below one.  On the $100$-worker ring the error stays above one through $B=50$ with constant momenta and through $B=100$ with the theorem schedule.  Once $B$ is large the error rises again; $B=5000$ leaves only $T=19$ updates.  The minimizing block lengths lie in $[200,625]$, except on the $100$-worker ring with constant momenta, where $B^\star=1000$ (\cref{tab:exp-block-length}).  Block lengths within a factor of ten of the best span a factor between $1.25$ and $16$ in $B$, about three in a typical configuration.

\paragraph{Consensus.}
At these block lengths, the ring and the torus produce nearly identical errors for every $M\le49$ (\cref{tab:exp-comparison}), although their spectral gaps differ by a factor of up to $28$.  For $M\le49$ and $\sigma^2=\zeta^2=1$, the consensus distance $\frac1M\sum_{i=1}^M\norm{x_i-\bar x}^2$ accounts for less than $10^{-4}$ of the final error, and \dadq{} behaves as centralized accelerated minibatch SGD with minibatch size $MB$.  On the ring, $B^\star$ grows from $250$ at $M=9$ ($\rho=0.16$) to $1000$ at $M=100$ ($\rho=1.3\cdot10^{-3}$); on the torus it remains between $200$ and $500$.  The block length prescribed by \cref{thm:distributed-gossip} is comparable to the measured minimizers at $M=4$, and exceeds them by more than an order of magnitude on the $100$-worker ring.

\end{document}